\title{M\"{o}bius Functions of Some Annular Noncrossing Objects}
\author{C.\ E.\ I.\ Redelmeier}

\documentclass[11pt]{article}

\usepackage{graphicx}
\usepackage{graphics}
\usepackage{amsfonts}
\usepackage{amssymb}
\usepackage{amsthm}
\usepackage{amsmath}
\usepackage{color}
\usepackage{cite}
\usepackage{enumitem}

\newtheorem{theorem}{Theorem}[section]
\newtheorem{lemma}[theorem]{Lemma}
\newtheorem{proposition}[theorem]{Proposition}

\theoremstyle{remark}
\newtheorem{remark}[theorem]{Remark}

\newtheorem{example}[theorem]{Example}

\theoremstyle{definition}
\newtheorem{definition}[theorem]{Definition}

\newtheoremstyle{case}{}{}{}{}{}{:}{ }{}
\theoremstyle{case}
\newtheorem{case}{Case}

\begin{document}

\maketitle

\begin{abstract}
We present the M\"{o}bius functions of several posets of annular noncrossing objects, namely a self-dual extension of the annular noncrossing permutations, minimal length annular partitioned permutations, and annular noncrossing partitons.
\end{abstract}

\section{Introduction}

Annular noncrossing objects are a natural generalization of noncrossing objects \cite{MR142470}, with applications in random matrix theory \cite{MR2052516, MR1492512}, second-order freeness \cite{MR2216446, MR2294222, MR2302524, MR3217665, redelmeier2015quaternionic}, and statistical physics \cite{MR1722987}.  There are a number of natural definitions of annular noncrossing objects leading to distinct sets of objects (unlike in the usual case, where the set of objects is the same regardless of whether they are defined on the disc or on the line, or whether the object is represented as a partition or a permutation), such as whether the cyclic ordering of elements in a block or the winding number around the annulus distinguishes objects.  In this paper, we consider a number of distinct sets of annular noncrossing objects.

Like the usual noncrossing objects, the various annular noncrossing objects may be arranged into partially ordered sets.  A natural construction on a poset is the M\"{o}bius function of the poset.  As with the usual noncrossing objects \cite{MR2266879}, the M\"{o}bius function is useful in the applications \cite{MR2302524}.

In this paper, we present the M\"{o}bius functions of the posets of several annular noncrossing objects.  The preliminaries, including notation and basic definitions, are given in Section~\ref{section: preliminaries}; however, many defintions and lemmas are included in most relevant sections.  The background on annular noncrossing permutations, including the M\"{o}bius function, is given in Section~\ref{section: permutations}.  In Section~\ref{section: self dual}, we give the M\"{o}bius function of a self-dual extension of the annular noncrossing permutations which has a smallest and largest element.  The values of this M\"{o}bius function appear in the second-order free moment-cumulant formula \cite{MR2302524}.  In Section~\ref{section: partitioned permutations} we present the M\"{o}bius function of the partitioned permutations which correspond to the highest-order terms of fluctuation of large random matrices \cite{MR2294222, MR2302524}.  In Section~\ref{section: partitions} we present the M\"{o}bius function of the annular noncrossing partitions \cite{MR1722987}, which record the elements which appear in the same block, but not their cyclic order.

\section{Preliminaries}

\label{section: preliminaries}

For integers \(m,n\) with \(0<m\leq n\), we denote \(\left\{1,\ldots,n\right\}\) by \(\left[n\right]\) and \(\left\{m,\ldots,n\right\}\) by \(\left[m,n\right]\).

We denote the number of elements in a finite set \(I\) by \(\left|I\right|\).

We denote the set of permutations on a set \(I\) by \(S\left(I\right)\), and \(S\left(\left[n\right]\right)\) by \(S_{n}\).  We will typically use cycle notation, and will often omit trivial cycles.  We will apply permutations right-to-left.  We denote the number of cycles in \(\pi\) by \(\#\left(\pi\right)\).

For integers \(n_{1},\ldots,n_{r}\), we denote
\[\tau_{n_{1},\ldots,n_{r}}:=\left(1,\ldots,n_{1}\right)\cdots\left(n_{1}+\cdots+n_{r-1},\ldots,n_{1}+\cdots+n_{r}\right)\]
and in particular, we will often denote
\[\tau:=\tau_{p,q}=\left(1,\ldots,p\right)\left(p+1,\ldots,p+q\right)\textrm{.}\]

\begin{definition}
A {\em partition} of a set \(I\) is a set \(\left\{U_{1},\ldots,U_{n}\right\}\) of nonempty subsets (or {\em blocks}) of \(I\) such that \(U_{i}\cap U_{j}=\emptyset\) for \(i\neq j\) and \(U_{1}\cup\cdots\cup U_{n}=S\).

The set of all partitions of a finite set \(I\) is denoted by \({\cal P}\left(I\right)\) (and \({\cal P}\left(\left[n\right]\right)\) is denoted \({\cal P}\left(n\right)\)).

For partitions \({\cal U},{\cal V}\in{\cal P}\left(I\right)\), we say that \({\cal U}\preceq{\cal V}\) if every block of \({\cal U}\) is contained in a block of \({\cal V}\).

We define the function \(\Pi:S\left(I\right)\rightarrow{\cal P}\left(I\right)\) by letting \(\Pi\left(\pi\right)\) be the set of orbits of \(\pi\).  Where there is no ambiguity, we will often use a permutation \(\pi\) to represent the partition \(\Pi\left(\pi\right)\).
\end{definition}

\begin{definition}
For a permutation \(\pi\in S\left(I\right)\), the permutation induced by \(\pi\) on \(J\subseteq I\), denoted \(\left.\pi\right|_{J}\), is the permutation such that
\[\left.\pi\right|_{J}\left(a\right)=\pi^{k}\left(a\right)\]
where \(k>0\) is the smallest integer such that \(\pi^{k}\left(a\right)\in J\).  This permutation may be cosntructed from \(\pi\) in cycle notation by deleting elements not in \(J\).

If \({\cal U}=\left\{U_{1},\ldots,U_{k}\right\}\in{\cal P}\left(I\right)\), then we define
\[\left.\pi\right|_{{\cal U}}:=\left(\left.\pi\right|_{U_{1}}\right)\cdots\left(\left.\pi\right|_{U_{k}}\right)\textrm{.}\]

\end{definition}

\subsection{Posets and the M\"{o}bius function}

We will use interval notation in posets: \(\left[x,y\right]=\left\{z:x\preceq z\preceq y\right\}\), \(\left(x,y\right)=\left\{z:x\prec z\prec y\right\}\), etc.

We will use the following standard results on posets and M\"{o}bius functions.  Proofs and additional information can be found in various sources, including \cite{MR1311922, MR2266879}
\begin{lemma}
If the intervals of posets \(\left[x_{1},y_{1}\right]\in P_{1}\) and \(\left[x_{2},y_{2}\right]\in P_{2}\) are isomorphic (i.e.\ there is a bijection \(f:\left[x_{1},y_{1}\right]\rightarrow\left[x_{2},x_{2}\right]\) such that \(z_{1}\preceq z_{2}\Leftrightarrow f\left(z_{1}\right)\preceq f\left(z_{2}\right)\)), and if \(\mu_{1}\) is the M\"{o}bius function on \(P_{1}\) and \(\mu_{2}\) is the M\"{o}bius function on \(P_{2}\), then \(\mu_{1}\left(x_{1},y_{1}\right)=\mu_{2}\left(x_{2},y_{2}\right)\).

If the intervals of posets \(\left[x_{1},y_{1}\right]\in P_{1}\) and \(\left[x_{2},y_{2}\right]\in P_{2}\) are dual (i.e.\ there is a bijection \(f:\left[x_{1},y_{1}\right]\rightarrow\left[x_{2},x_{2}\right]\) such that \(z_{1}\preceq z_{2}\Leftrightarrow f\left(z_{1}\right)\succeq f\left(z_{2}\right)\)), then if \(\mu_{1}\) is the M\"{o}bius function on \(P_{1}\) and \(\mu_{2}\) is the M\"{o}bius function on \(P_{2}\), then again \(\mu_{1}\left(x_{1},y_{1}\right)=\mu_{2}\left(x_{2},y_{2}\right)\).

Let \(P_{1},\ldots,P_{n}\) are posets, and let the M\"{o}bius function of \(P_{k}\) be \(\mu_{k}\) and \(x_{k},y_{k}\in P_{k}\) for \(k=1,\ldots,n\).  Let \(\mu\) be the M\"{o}bius function of the product poset \(P_{1}\times\cdot\times P_{n}\).  Then
\[\mu\left(\left(x_{1},\ldots,x_{n}\right),\left(y_{1},\ldots,y_{n}\right)\right)=\mu_{1}\left(x_{1},y_{2}\right)\cdots\mu_{n}\left(x_{n},y_{n}\right)\textrm{.}\]
\end{lemma}

\subsection{Catalan numbers and a second-order generalization}

\begin{definition}
The \(n\)th Catalan number is
\[C_{n}:=\frac{1}{n+1}\binom{2n}{n}\textrm{.}\]
The generating function for the Catalan numbers is
\[C\left(x\right):=\sum_{n=0}^{\infty}C_{n}x^{n}=\frac{1-\sqrt{1-4x}}{2x}\textrm{.}\]
\end{definition}

A relevant second-order generalization of the Catalan numbers (from \cite{MR1761777, MR1959915}) is:
\[\gamma_{p,q}=\frac{2}{p+q}\frac{\left(2p-1\right)!}{\left(p-1\right)!^{2}}\frac{\left(2q-1\right)!}{\left(q-1\right)!^{2}}\textrm{.}\]

\subsection{Cartographic machinery}

The geometric motivation for the following constructions is described in such sources as \cite{MR0404045, MR1603700, MR2036721}.

\begin{definition}
The {\em Kreweras complement} of a permutation \(\pi\) (with respect to a permutation \(\tau\)) is defined as
\[\mathrm{Kr}_{\tau}\left(\pi\right):=\pi^{-1}\tau\textrm{.}\]
The inverse of this function is
\[\mathrm{Kr}_{\tau}^{-1}\left(\pi\right)=\tau\pi^{-1}\textrm{.}\]
If \(\tau\) is understood (typically \(\tau_{p,q}\)), it may be omitted from the notation.
\end{definition}

\begin{definition}
For \(\pi,\rho\in S\left(I\right)\), we say that \(\rho\) is noncrossing on \(\pi\) if
\[\#\left(\pi\right)+\#\left(\rho\right)+\#\left(\mathrm{Kr}_{\pi}\left(\rho\right)\right)=\left|I\right|+2\#\left(\langle\pi,\rho\rangle\right)\]
where \(\#\left(\langle\pi,\rho\rangle\right)\) is the number of orbits in \(I\) of the subgroup generated by \(\pi\) and \(\rho\).
The set of \(\rho\) which are noncrossing on \(\pi\) is denoted \(S_{\mathrm{nc}}\left(\pi\right)\).

If, in addition, \(\Pi\left(\rho\right)\preceq\Pi\left(\pi\right)\), we say that \(\rho\) is disc-noncrossing on \(\pi\), and that \(\rho\preceq\pi\).  The set of \(\rho\) which are disc-noncrossing on \(\pi\) is denoted \(S_{\mathrm{disc-nc}}\left(\pi\right)\).
\end{definition}

For any permutation \(\pi\), the set \(S_{\mathrm{disc-nc}}\left(\pi\right)\) is a lattice, i.e.\ any set of elements has a greatest lower bound and a least upper bound (see, e.g., \cite{MR2266879}).

For \(\rho,\sigma\in S_{\mathrm{disc-nc}}\left(\pi\right)\) with \(\rho\preceq\sigma\), the M\"{o}bius function (which we will denote by \(\mu\)) is
\begin{equation}
\mu\left(\rho,\sigma\right)=\prod_{U\in\Pi\left(\mathrm{Kr}_{\sigma}\left(\rho\right)\right)}\left(-1\right)^{\left|U\right|-1}C_{\left|U\right|-1}\textrm{.}
\label{equation: disc noncrossing}
\end{equation}
The proof is as in \cite{MR2266879}, Lectures~9 and 10.

If \(\tau\) has \(1\) cycle, \(\pi\) being noncrossing is equivalent to the following:
\begin{lemma}[Biane]
Let \(\tau\in S\left(I\right)\) have one cycle.  Then \(\pi\in S_{\mathrm{nc}}\left(\tau\right)\) iff neither of the following conditions occurs:
\begin{enumerate}
	\item There are \(a,b,c\in I\) such that \(\left.\tau\right|_{\left\{a,b,c\right\}}=\left(a,b,c\right)\) and \(\left.\pi\right|_{\left\{a,b,c\right\}}=\left(a,c,b\right)\).
	\item There are \(a,b,c,d\in I\) such that \(\left.\tau\right|_{\left\{a,b,c,d\right\}}=\left(a,b,c,d\right)\) and \(\left.\pi\right|_{\left\{a,b,c,d\right\}}=\left(a,c\right)\left(b,d\right)\).
\end{enumerate}
\label{lemma: disc noncrossing}
\end{lemma}
See \cite{MR1475837}.

\section{Noncrossing Permutations on the Annulus}

\label{section: permutations}

In this section we collect the relevant definitions and results on annular noncrossing permutations that will be used througout the paper, and record the M\"{o}bius function for the noncrossing permutations on the annulus.

\begin{definition}
The set of permutations which are noncrossing on \(\tau_{p,q}\) is denoted \(S_{\mathrm{nc}}\left(p,q\right)\).  The set of permutations \(\pi\in S_{\mathrm{nc}}\left(p,q\right)\) with \(\Pi\left(\pi\right)\npreceq\Pi\left(\tau\right)\) (i.e.\ the permutations that connect the two discs of the annulus) is denoted \(S_{\mathrm{ann-nc}}\left(p,q\right)\).

A cycle of \(\pi\in S_{\mathrm{nc}}\left(p,q\right)\) which contains elements from both cycles of \(\tau_{p,q}\) will be called a {\em bridge}.

We will use the partial order defined in Section~\ref{section: preliminaries}.  We may denote the smallest element of this poset \(\left(1\right)\cdots\left(p+q\right)\) by \(0\).  This poset does not have a largest element.

In this paper, we denote the M\"{o}bius function of this poset on \(S_{\mathrm{nc}}\left(p,q\right)\) simply by \(\mu\), since it is consistent with all the other M\"{o}bius functions discussed in this paper.
\end{definition}

The following result from \cite{MR2052516} gives conditions on a permutation \(\pi\in S_{p+q}\) equivalent to it being noncrossing on \(\tau_{p,q}\).

\begin{lemma}[Mingo, Nica]
Let \(\tau:=\tau_{p,q}\) and let \(\pi\in S_{p+q}\)

The annular-nonstandard conditions are:
\begin{enumerate}
	\item There are \(a,b,c\in\left[p+q\right]\) with \(\left.\tau\right|_{\left\{a,b,c\right\}}\left(a,b,c\right)\) and \(\left.\pi\right|_{\left\{a,b,c\right\}}\left(a,c,b\right)\).
\label{item: ans1}
	\item There are \(a,b,c,d\in\left[p+q\right]\) with \(\left.\tau\right|_{\left\{a,b,c,d\right\}}\left(a,b\right)\left(c,d\right)\) and \(\left.\pi\right|_{\left\{a,b,c,d\right\}}\left(a,c,b,d\right)\).
\label{item: ans2}
\end{enumerate}

For \(x,y\) in different cycles of \(\tau\), we let
\[\lambda_{x,y}:=\left(\tau\left(x\right),\tau^{2}\left(x\right),\ldots,\tau^{-1}\left(x\right),\tau\left(y\right),\tau^{2}\left(y\right),\ldots,\tau^{-1}\left(y\right)\right)\textrm{.}\]
The annular-crossing conditions are:
\begin{enumerate}
	\item There are \(a,b,c,d\in\left[p+q\right]\) with \(\left.\tau\right|_{\left\{a,b,c,d\right\}}\left(a,b,c,d\right)\) and \(\left.\pi\right|_{\left\{a,b,c,d\right\}}=\left(a,c\right)\left(b,d\right)\).
\label{item: ac1}
	\item There are \(a,b,c,x,y\in\left[p+q\right]\) such that \(\left.\lambda_{x,y}\right|_{\left\{a,b,c\right\}}=\left(a,b,c\right)\) and \(\left.\pi\right|_{\left\{a,b,c,x,y\right\}}=\left(a,c,b\right)\left(x,y\right)\).
\label{item: ac2}
	\item There are \(a,b,c,d,x,y\in\left[p+q\right]\) such that and \(\left.\lambda_{x,y}\right|_{\left\{a,b,c,d\right\}}=\left(a,b,c,d\right)\) and \(\left.\pi\right|_{\left\{a,b,c,d,x,y\right\}}=\left(a,c\right)\left(b,d\right)\left(x,y\right)\).
\label{item: ac3}
\end{enumerate}

Iff none of the above conditions occurs, then \(\pi\in S_{\mathrm{nc}}\left(p,q\right)\).
\end{lemma}

\begin{proposition}
For \(\pi,\rho\in S_{\mathrm{nc}}\left(p,q\right)\) with \(\pi\preceq\rho\),
\begin{equation}
\mu\left(\pi,\rho\right)=\prod_{U\in\Pi\left(\mathrm{Kr}_{\rho}\left(\pi\right)\right)}\left(-1\right)^{\left|U\right|-1}C_{\left|U\right|-1}\textrm{.}
\label{equation: base}
\end{equation}
\end{proposition}
\begin{proof}
The poset \(\left[\pi,\rho\right]\) is the product over blocks \(V\in\Pi\left(\rho\right)\) of the poset of disc-noncrossing permutations in the interval \(\left[\left.\pi\right|_{V},\left.\rho\right|_{V}\right]\), from which (\ref{equation: base}) follows.
\end{proof}

The M\"{o}bius function of a natural poset on a definition of the annular noncrossing configurations which includes the winding number of bridges would have the same values, since the possible noncrossing configurations on each block of \(\rho\) would be the disc-noncrossing configurations on that block.

\section{A self-dual extension of the annular noncrossing permutations}

\label{section: self dual}

In this poset, we extend \(S_{\mathrm{nc}}\left(p,q\right)\) so that it is self-dual and has a largest element.  The values of this M\"{o}bius function appear in the expression for second-order free cumulants \cite{MR2302524}.

\begin{definition}
Let \(p,q>0\) be integers, and let \(S_{\mathrm{nc-sd}}\left(p,q\right)\) be the disjoint union of \(S_{\mathrm{ann-nc}}\left(p,q\right)\) with two copies of \(S_{\mathrm{disc-nc}}\left(p,q\right)\), which in this context we will denote \(S_{\mathrm{disc-nc}}\left(p,q\right)\) and \(\hat{S}_{\mathrm{disc-nc}}\left(p,q\right)\).  If an element of the first copy is denoted \(\pi\), we will denote the corresponding element of the second copy by \(\hat{\pi}\) (while an undecorated \(\pi\in S_{\mathrm{nc-sd}}\left(p,q\right)\) may belong to any of the three parts of the poset).  We define \(\widehat{\mathrm{Kr}}:S_{\mathrm{nc-sd}}\left(p,q\right)\rightarrow S_{\mathrm{nc-sd}}\left(p,q\right)\) to take \(\pi\) to an element whose underlying permutation is \(\mathrm{Kr}\left(\pi\right)\) and taking elements of \(S_{\mathrm{disc-nc}}\left(p,q\right)\) to elements of \(\hat{S}_{\mathrm{disc-nc}}\left(p,q\right)\) and {\em vice versa}.

If \(\pi,\rho\in S_{\mathrm{disc-nc}}\left(p,q\right)\cup S_{\mathrm{ann-nc}}\left(p,q\right)\), then \(\pi\preceq \rho\) if \(\pi\preceq\rho\) in the poset of Section~\ref{section: permutations}.  If \(\pi\in S_{\mathrm{nc-sd}}\left(p,q\right)\) and \(\rho\in S_{\mathrm{disc-nc}}\left(p,q\right)\), then \(\pi\preceq\hat{\rho}\) if \(\mathrm{Kr}\left(\pi\right)\succeq\mathrm{Kr}\left(\rho\right)\) in the poset of Section~\ref{section: permutations}.  (It can be easily verified that this creates a poset from the fact that \(S_{\mathrm{disc-nc}}\left(p,q\right)\cup S_{\mathrm{ann-nc}}\left(p,q\right)\) is poset, and \(\widehat{\mathrm{Kr}}\) and \(\widehat{\mathrm{Kr}}^{-1}\) are order-reversing bijections.)

We denote the smallest element, \(\left(1\right)\cdots\left(p+q\right)\), by \(0\).  We denote the largest element, \(\widehat{\left(1,\ldots,p\right)\left(p+1,\ldots,p+q\right)}\), by \(1\).

We will denote the M\"{o}bius function on \(S_{\mathrm{nc-sd}}\left(p,q\right)\) by \(\mu_{\mathrm{sd}}\).
\end{definition}

\begin{remark}
We note that \(S_{\mathrm{nc-sd}}\left(p,q\right)\) is not a lattice.  In \(S_{\mathrm{nc-sd}}\left(1,2\right)\), for example, both \(\left(1,2\right)\left(3\right)\) and \(\left(1,3\right)\left(2\right)\) are covered by both \(\left(1,2,3\right)\) and \(\left(1,3,2\right)\), so they do not have a unique supremum.
\end{remark}

We collect several technical results in Lemma~\ref{lemma: technical}.  Roughly, these results show how an element of \(S_{\mathrm{ann-nc}}\left(p,q\right)\) can be constructed from an element of \(S_{\mathrm{disc-nc}}\left(p,q\right)\) by choosing the outside faces of the diagrams on the two discs, then connecting cycles to form bridges in a noncrossing way.  (See Example~\ref{example: annuli} for the intuition behind the results.)

\begin{lemma}
Let \(\pi\in S_{\mathrm{ann-nc}}\left(p,q\right)\).  Here \(\tau=\tau_{p,q}\) and \(\mathrm{Kr}=\mathrm{Kr}_{\tau}\).  Let \(\pi\in S_{\mathrm{an-nc}}\left(p,q\right)\) and let \(\pi_{0}:=\left.\pi\right|_{\tau}\in S_{\mathrm{disc-nc}}\left(p,q\right)\).  Then:
\begin{enumerate}
	\item Any bridge of \(\pi\) is of the form \(\left(a_{1},\ldots,a_{r},b_{1},\ldots,b_{s}\right)\), where \(a_{1},\ldots,a_{r}\in\left[p\right]\) and \(b_{1},\ldots,b_{s}\in\left[p+1,p+q\right]\).  Therefore there is exactly one element \(a\) of a bridge in \(\left[p\right]\) (resp.\ \(\left[p+1,p+q\right]\)) with \(\pi\left(a\right)\in\left[p+1,p+q\right]\) (resp.\ \(\left[p\right]\)).
\label{item: bridge}
	\item If all cycles of \(\pi\) are bridges, they are of the form, for some \(a\in\left[p\right]\) and some \(b\in\left[p+1,p+q\right]\), \(\left(a,\tau\left(a\right),\ldots,\tau^{r-1}\left(a\right),b,\tau\left(b\right),\ldots,\tau^{s-1}\left(b\right)\right)\).  If, following one cycle of \(\tau\), we encounter the bridges of \(\pi\) in a given (cyclic) order, then they will be encountered in the reverse (cyclic) order following the other cycle of \(\tau\).  Any \(\pi\) satisfying these two conditions is in \(S_{\mathrm{ann-nc}}\left(p,q\right)\).
\label{item: all bridges}
	\item The set \(I\) of elements in bridges of \(\mathrm{Kr}\left(\pi\right)\) (resp. \(\mathrm{Kr}^{-1}\left(\pi\right)\)) comprise two orbits of \(\mathrm{Kr}\left(\pi_{0}\right)\) (resp.\ \(\mathrm{Kr}^{-1}\left(\pi_{0}\right)\)), one contained in each of \(\left[p\right]\) and \(\left[p+1,p+q\right]\).  (These may be viewed as the ``outside faces'' of the configuration; see Figure~\ref{figure: outside faces}.)

The \(a\in\left[p\right]\) (resp.\ \(\left[p+1,p+q\right]\)) with \(\pi^{-1}\left(a\right)\in\left[p+1,p+q\right]\) (resp.\ \(\left[p\right]\)) are contained in \(I\).
\label{item: outside faces}
	\item For \(\rho\in S_{\mathrm{disc-nc}}\left(p,q\right)\), \(\pi\preceq\hat{\rho}\) when \(\pi_{0}\preceq\rho\) and the bridges of \(\pi\) are contained in two cycles of \(\rho\) (necessarily one of which will be contained in \(\left[p\right]\) and the other in \(\left[p+1,p+q\right]\)).
\label{item: partial order}
	\item Let \(\rho\in S_{\mathrm{disc-nc}}\left(p,q\right)\), and let \(J\) be the set of elements in two cycles of \(\rho\), one in each cycle of \(\tau\).  Let \(\sigma_{1}\in S_{\mathrm{ann-nc}}\left(\left.\rho\right|_{J}\right)\) and \(\sigma_{2}\in S_{\mathrm{disc-nc}}\left(\left.\rho\right|_{\left[p+q\right]\setminus J}\right)\).  Then \(\sigma:=\sigma_{1}\sigma_{2}\in S_{\mathrm{nc}}\left(p,q\right)\)
\label{item: piecewise}
	\item Let \(\rho\in S_{\mathrm{disc-nc}}\left(p,q\right)\) with \(\hat{\rho}\succeq\pi\).  Let \(I\) be the set of elements in bridges of \(\mathrm{Kr_{\tau}}^{-1}\left(\pi\right)\), let \(J\) be the set of elements in the two cycles of \(\rho\) containing the bridges of \(\pi\), and let \(K:=I\cap J\) (here and in the remaining parts of this lemma).  The elements of \(K\) form two cycles in \(\mathrm{Kr}_{\rho}^{-1}\left(\pi_{0}\right)\), one contained in each cycle of \(\tau\).
\label{item: selected cycles}
	\item Any cycle of \(\mathrm{Kr}_{\rho}^{-1}\left(\pi\right)\) is either fully contained in \(K\) or outside of it.
\label{item: cycle partition}
	\item If a cycle of \(\mathrm{Kr}_{\rho}^{-1}\left(\pi\right)\) is not in \(K\), then it also appears as a cycle of \(\mathrm{Kr}_{\rho}^{-1}\left(\pi_{0}\right)\).
\label{item: spare cycles}
	\item The cycles of \(\mathrm{Kr}_{\rho}^{-1}\left(\pi\right)\) in \(K\) are all bridges; in fact, for any \(\pi_{0}\preceq\rho\in S_{\mathrm{disc-nc}}\left(p,q\right)\), any choice of noncrossing two-element bridges on \(\left.\tau\right|_{K}\) is induced by a unique \(\pi\preceq\hat{\rho}\) with \(\left.\pi\right|_{\tau}=\pi_{0}\) and this choice of outside faces of \(\pi_{0}\).
\label{item: skeleton}
\end{enumerate}
\label{lemma: technical}
\end{lemma}
\begin{proof}
\ref{item: bridge}.: This follows immediately from annular-nonstandard condition~\ref{item: ans2}.

\ref{item: all bridges}.: Consider a bridge of \(\pi\) \(\left(a_{1},\ldots,a_{r},b_{1},\ldots,b_{s}\right)\) with \(a_{1},\ldots,a_{r}\in\left[p\right]\), \(b_{1},\ldots,b_{s}\in\left[p+1,p+q\right]\).  If there is an \(x\in\left[p\right]\) which is not in this bridge, but which appears between \(a_{1}\) and \(a_{r}\) in \(\tau\), then there must be a \(y\in\left[p+1,p+q\right]\) in the same cycle of \(\pi\) as \(x\).  Then \(\left.\pi\right|_{\left\{a_{1},a_{r},b_{1}\right\}}=\left(a_{1},a_{r},b_{1}\right)\) but \(\left.\lambda_{x,y}\right|_{\left\{a_{1},a_{r},b_{1}\right\}}=\left(a_{1},b_{1},a_{r}\right)\), so by annular-crossing condition~\ref{item: ac2} \(\pi\notin S_{\mathrm{ann-nc}}\left(p,q\right)\).  Likewise if an \(x\in\left[p+1,p+q\right]\) appears between \(b_{1}\) and \(b_{s}\) in \(\tau\).

If the cycles of \(\pi\) do not intersect the two cycles of \(\tau\) in reverse cyclic order, then when we construct one of the \(\lambda_{x,y}\), it will be possible to find cycles of \(\pi\) with elements \(a\) and \(b\) (resp.\ \(c\) and \(d\)) from the two cycles of \(\tau\) such that \(\left.\lambda_{x,y}\right|_{\left\{a,b,c,d\right\}}=\left(a,c,b,d\right)\), satisfying annular-crossing condition \ref{item: ac3}.

Conversely, let \(\pi\) have the form described in the statement.  Then if \(\pi\left(a\right)\) is in the same cycle of \(\tau\) as \(a\), \(\mathrm{Kr}\left(\pi\right)\left(a\right)=\pi^{-1}\tau\left(a\right)=\tau^{-1}\tau\left(a\right)=a\).  If \(\pi\left(a\right)\) is not in the same cycle of \(\tau\) as \(a\), then \(\tau\left(a\right)\) is the first element of the next bridge of \(\pi\).  Then \(\pi^{-1}\tau\left(a\right)\) is the last element of this bridge in the other cycle of \(\tau\) (and is thus not equal to \(a\)).  Since the next cycle of \(\pi\) along the second cycle of \(\tau\) is the original bridge of \(\pi\) containing \(a\), \(\left(\pi^{-1}\tau\right)^{2}\left(a\right)=a\).  Thus \(\mathrm{Kr}\left(\pi\right)\) consists of two-element bridges and one-element cycles.  This construction also shows that \(\pi\) and \(\mathrm{Kr}\left(\pi\right)\) have the same number of bridges, say \(k\).  We calculate that \(\chi\left(\pi\right)=2+k+\left(p+q-k\right)-\left(p+q\right)=2\), so \(\pi\in S_{\mathrm{ann-nc}}\left(p,q\right)\).

\ref{item: outside faces}.: We demonstrate the result for \(\mathrm{Kr}^{-1}\); the proof for \(\mathrm{Kr}\) is similar.  For every \(a\in\left[p\right]\) with \(\pi^{-1}\left(a\right)\in\left[p+1,q\right]\), \(\mathrm{Kr}^{-1}\left(\pi\right)\left(a\right)=\tau\pi^{-1}\left(a\right)\in\left[p+1,p+q\right]\), and, conversely, if \(a\in\left[p\right]\) and \(\mathrm{Kr}^{-1}\left(\pi\right)\left(a\right)\in\left[p+1,p+q\right]\), then \(\pi^{-1}\left(a\right)\in\left[p+1,p+q\right]\).  By part~\ref{item: bridge}, a bridge has exactly one element it maps from \(\left[p\right]\) to \(\left[p+1,p+q\right]\) and exactly one element it maps from \(\left[p+1,p+q\right]\) to \(\left[p\right]\), so there must be the same number \(k\) of bridges in \(\pi\) as in \(\mathrm{Kr}^{-1}\left(\pi\right)\).

We know that \(\#\left(\pi\right)+\#\left(\mathrm{Kr}^{-1}\left(\pi\right)\right)=p+q\) and \(\#\left(\pi_{0}\right)+\#\left(\mathrm{Kr}^{-1}\left(\pi_{0}\right)\right)=p+q+2\).  By part~\ref{item: bridge}., \(\#\left(\pi_{0}\right)=\#\left(\pi\right)+k\) (since each bridge induces a cycle on each cycle of \(\tau\), and all other cycles induce a cycle on only one cycle of \(\tau\)), so \(\#\left(\mathrm{Kr}^{-1}\left(\pi_{0}\right)\right)=\#\left(\mathrm{Kr}^{-1}\left(\pi\right)\right)-k+2\).  The \(\#\left(\mathrm{Kr}^{-1}\left(\pi\right)\right)-k\) cycles of \(\mathrm{Kr}^{-1}\left(\pi\right)\) which are not bridges appear as cycles of \(\mathrm{Kr}^{-1}\left(\pi_{0}\right)\): if \(\tau\pi^{-1}\left(a\right)\) is in the same cycle of \(\tau\) as \(a\), then it is equal to \(\tau\pi_{0}^{-1}\left(a\right)\).  The elements of the bridges of \(\mathrm{Kr}^{-1}\left(\pi\right)\) must fall into at least two cycles of \(\mathrm{Kr}^{-1}\left(\pi_{0}\right)\), so they must comprise the two remaining cycles of \(\mathrm{Kr}^{-1}\left(\pi_{0}\right)\).

This also implies that all the \(a\in\left[p\right]\) (resp.\ \(\left[p+1,p+q\right]\)) with \(\pi^{-1}\left(a\right)\in\left[p+1,p+q\right]\) (resp.\ \(\left[p\right]\)) are in bridges of \(\mathrm{Kr}^{-1}\left(\pi\right)\), so they appear in the same cycle of \(\mathrm{Kr}^{-1}\left(\pi_{0}\right)\).

\ref{item: partial order}.: If \(\pi\preceq\hat{\rho}\), then \(\pi_{0}\preceq\pi\preceq\hat{\rho}\).  Furthermore, by part~\ref{item: outside faces}, the bridges of \(\pi=\mathrm{Kr}^{-1}\left(\mathrm{Kr}\left(\pi\right)\right)\) are contained in two cycles of \(\mathrm{Kr}^{-1}\left(\left.\mathrm{Kr}\left(\pi\right)\right|_{\tau}\right)\).  Because \(\mathrm{Kr}\left(\rho\right)\preceq\mathrm{Kr}\left(\pi\right)\), and each cycle of \(\mathrm{Kr}\left(\rho\right)\) is contained in a cycle of \(\tau\), \(\mathrm{Kr}\left(\rho\right)\preceq\left.\mathrm{Kr}\left(\pi\right)\right|_{\tau}\).  Thus each of the two cycles of \(\mathrm{Kr}^{-1}\left(\left.\mathrm{Kr}\left(\pi\right)\right|_{\tau}\right)\) containing the bridges of \(\pi\) is contained in a cycle of \(\mathrm{Kr}^{-1}\left(\mathrm{Kr}\left(\rho\right)\right)=\rho\).

Conversely, if \(\pi_{0}\preceq\rho\) and the bridges of \(\pi=\mathrm{Kr}^{-1}\left(\mathrm{Kr}\left(\pi\right)\right)\) are contained in two cycles of \(\rho=\mathrm{Kr}^{-1}\left(\mathrm{Kr}\left(\rho\right)\right)\), then each of these cycles must contain the cycles of \(\mathrm{Kr}^{-1}\left(\left.\mathrm{Kr}\left(\pi\right)\right|_{\tau}\right)\) that have the same elements as the bridges of \(\pi\).  For any \(a\) which is not an element of a bridge of \(\pi\), \(\mathrm{Kr}^{-1}\left(\left.\mathrm{Kr}\left(\pi\right)\right|_{\tau}\right)\left(a\right)=\tau\left.\mathrm{Kr}\left(\pi\right)\right|_{\tau}^{-1}\left(a\right)=\tau\mathrm{Kr}\left(\pi\right)^{-1}\left(a\right)=\tau\tau^{-1}\pi\left(a\right)=\pi_{0}\left(a\right)\) (the last step because this element is in the same cycle of \(\tau\) as \(a\)).  Thus the rest of the cycles of \(\mathrm{Kr}^{-1}\left(\left.\mathrm{Kr}\left(\pi\right)\right|_{\tau}\right)\) appear as cycles of \(\pi_{0}\).  Because \(\pi_{0}\preceq\rho\), any such cycle must be contained in a cycle of \(\rho\).  So
\[\mathrm{Kr}^{-1}\left(\left.\mathrm{Kr}\left(\pi\right)\right|_{\tau}\right)\preceq\rho\Rightarrow\mathrm{Kr}\left(\rho\right)\preceq\left.\mathrm{Kr}\left(\pi\right)\right|_{\tau}\Rightarrow\mathrm{Kr}\left(\rho\right)\preceq\mathrm{Kr}\left(\pi\right)\Rightarrow\pi\preceq\hat{\rho}\]
as desired.

\ref{item: piecewise}.: Each of the annular-nonstandard and annular-noncrossing conditions can be checked.  We note that if \(L\subseteq\left[p+q\right]\) is contained in a cycle of \(\sigma\), then \(\left.\rho\right|_{L}=\left.\tau\right|_{L}\) (since \(L\) is either contained in a single cycle of \(\rho\) or in \(J\), the union of one cycle from each of the two different cycles of \(\tau\)).  Any annular-nonstandard or annular crossing condition in \(\sigma\) implies the same annular-nonstandard or annular-noncrossing condition in \(\sigma_{1}\) or \(\sigma_{2}\) (if the elements satisfying the condition are contained in \(J\) or in \(\left[p+q\right]\setminus J\)), or a possibly different condition in \(\rho\) (if the elements satisfying the condition are in both \(J\) and \(\left[p+q\right]\setminus J\)).  The cases requiring a different condition are (\(x\) and \(y\) are elements of a bridge and hence in \(\sigma_{1}\)):
\begin{itemize}
	\item if annular-crossing condition~\ref{item: ac2} occurs in \(\sigma\) and \(\left(a,c,b\right)\) is in \(\sigma_{2}\), then it is contained in one orbit of \(\tau\) (annular-nonstandard condition~\ref{item: ans1}), and
	\item if annular-crossing condition~\ref{item: ac3} occurs in \(\sigma\), then if \(\left(a,c\right)\) and \(\left(b,d\right)\) are both from \(\sigma_{2}\) then annular-crossing condition~\ref{item: ac1} occurs in \(\sigma_{2}\), and if one occurs in \(\sigma_{1}\), then it is contained in an orbit of \(\rho\) with one of \(x\) or \(y\), so annular-crossing condition~\ref{item: ac1} occurs in \(\rho\).
\end{itemize}

\ref{item: selected cycles}.: Since \(\pi_{0}\preceq\rho\), \(\mathrm{Kr}_{\rho}^{-1}\left(\pi_{0}\right)\preceq\rho\).  Let \(\varphi:=\mathrm{Kr}_{\tau}^{-1}\left(\pi_{0}\right)=\tau\pi_{0}^{-1}\).  Since \(\pi_{0}\preceq\rho\), \(\varphi\succeq\mathrm{Kr}_{\tau}^{-1}\left(\rho\right)\), and so \(\varphi\succeq\mathrm{Kr}_{\varphi}\left(\mathrm{Kr}_{\tau}^{-1}\left(\rho\right)\right)=\left(\tau\rho^{-1}\right)^{-1}\tau\pi_{0}^{-1}=\rho\pi_{0}^{-1}=\mathrm{Kr}_{\rho}^{-1}\left(\pi_{0}\right)\).  Thus any cycle of \(\mathrm{Kr}_{\rho}^{-1}\left(\pi_{0}\right)\) is contained in the intersection of a cycle of \(\rho\) and a cycle of \(\mathrm{Kr}_{\tau}^{-1}\left(\pi_{0}\right)\).

Conversely, we show that the part of \(K\) in each of the cycles of \(\tau\) consists of one cycle of \(\mathrm{Kr}_{\rho}^{-1}\left(\pi_{0}\right)=\rho\pi_{0}^{-1}\in S_{\mathrm{disc-nc}}\left(p,q\right)\) by showing that \(\mathrm{Kr}_{\rho}^{-1}\left(\pi_{0}\right)\succeq\rho\wedge\mathrm{Kr}_{\tau}^{-1}\left(\pi_{0}\right)\), the part of \(K\) in a given cycle of \(\tau\) being a block of \(\rho\wedge\mathrm{Kr}_{\tau}^{-1}\left(\pi\right)\).  Equivalently, we show \(\mathrm{Kr}_{\tau}\left(\rho\pi_{0}^{-1}\right)\preceq\mathrm{Kr}_{\tau}\left(\rho\wedge\mathrm{Kr}_{\tau}^{-1}\left(\pi_{0}\right)\right)=\mathrm{Kr}_{\tau}\left(\rho\right)\vee\pi_{0}\): since \(\mathrm{Kr}_{\tau}\left(\rho\pi_{0}^{-1}\right)=\pi_{0}\rho^{-1}\tau\) takes any element to another element of the same orbit of the subgroup generated by \(\pi_{0}\) and \(\rho^{-1}\tau=\mathrm{Kr}_{\tau}\left(\rho\right)\), it must be smaller than their supremum.

\ref{item: cycle partition}.: This will follow from parts~\ref{item: selected cycles}.\ and \ref{item: spare cycles}.\ (since \(K\) consists of complete cycles of \(\mathrm{Kr}_{\rho}^{-1}\left(\pi_{0}\right)\), so does the remainder \(\left[p+q\right]\setminus K\), which part~\ref{item: spare cycles}.\ shows are complete cycles of \(\mathrm{Kr}_{\rho}^{-1}\left(\pi\right)\)).

\ref{item: spare cycles}.: Let \(a\notin K\).  If \(a\) is not in the two cycles of \(\rho\) containing the bridges of \(\pi\), then \(\pi^{-1}\left(a\right)=\pi_{0}^{-1}\left(a\right)\), so \(\rho\pi^{-1}\left(a\right)=\rho\pi_{0}^{-1}\left(a\right)=\mathrm{Kr}_{\rho}^{-1}\left(\pi_{0}\right)\left(a\right)\).  If \(a\) is not in the outside faces of \(\mathrm{Kr}_{\tau}^{-1}\left(\pi_{0}\right)\), then it is not in the bridges of \(\mathrm{Kr}_{\tau}^{-1}\left(\pi\right)\), so \(\tau\pi^{-1}\left(a\right)\) and hence \(\pi^{-1}\left(a\right)\) must be in the same cycle of \(\tau\) as \(a\), and so again \(\rho\pi^{-1}\left(a\right)=\rho\pi_{0}^{-1}\left(a\right)=\mathrm{Kr}_{\rho}^{-1}\left(\pi_{0}\right)\left(a\right)\).

\ref{item: skeleton}.: Let \(J\) be the set of elements in the two cycles of \(\rho\) that contain the bridges of \(\pi\).  We note that unless \(\left.\rho\right|_{J}=\left.\tau\right|_{J}\), it is possible to find three elements which satisfy annular-nonstandard condition~\ref{item: ans1}.  So \(\left.\pi\right|_{J}\in S_{\mathrm{ann-nc}}\left(\left.\rho\right|_{J}\right)\).  The elements of the bridges of \(\mathrm{Kr}_{\rho}^{-1}\left(\pi\right)\) are the elements of the outside faces of \(\mathrm{Kr}_{\rho}^{-1}\left(\pi_{0}\right)\), which contain the \(a\in\left[p\right]\) with \(\pi^{-1}\left(a\right)\in\left[p+1,p+q\right]\) (resp.\ the \(a\in\left[p+1,p+q\right]\) with \(\pi^{-1}\left(a\right)\in\left[p\right]\)), and hence intersect with the outside faces of \(\mathrm{Kr}_{\tau}^{-1}\left(\pi\right)\).  Since they also intersect with \(J\), they must comprise \(K\).

Let \(\sigma_{1}\in S_{\mathrm{ann-nc}}\left(\left.\tau\right|_{K}\right)\) with every cycle a bridge, let \(\sigma_{2}=\left.\left(\rho\pi_{0}^{-1}\right)\right|_{\left[p+q\right]\setminus K}\), and let \(\sigma:=\sigma_{1}\sigma_{2}\).  Then by part~\ref{item: piecewise}., \(\sigma\in S_{\mathrm{ann-nc}}\left(p,q\right)\), as is \(\pi:=\mathrm{Kr}_{\rho}\left(\sigma\right)\).

We now show that \(\left.\pi\right|_{\tau}=\pi_{0}\) by showing that \(\sigma\) is \(\mathrm{Kr}_{\rho}^{-1}\) of a permutation with this property, as described in part~\ref{item: outside faces}.  On \(J\), the elements of the bridges of \(\sigma\) comprise two cycles of \(\mathrm{Kr}_{\rho}^{-1}\left(\pi_{0}\right)\), while the remaining cycles are those of \(\mathrm{Kr}_{\rho}^{-1}\left(\pi_{0}\right)\) (since they do not connect the cycles of \(\tau\)).  Likewise, outside of \(J\), the cycles of \(\sigma\) are the same as those of \(\mathrm{Kr}_{\rho}^{-1}\left(\pi_{0}\right)\).  

Uniqueness follows from the invertibility of \(\mathrm{Kr}_{\rho}\).
\end{proof}

\begin{figure}
\label{figure: outside faces}
\centering
\begin{tabular}{cc}
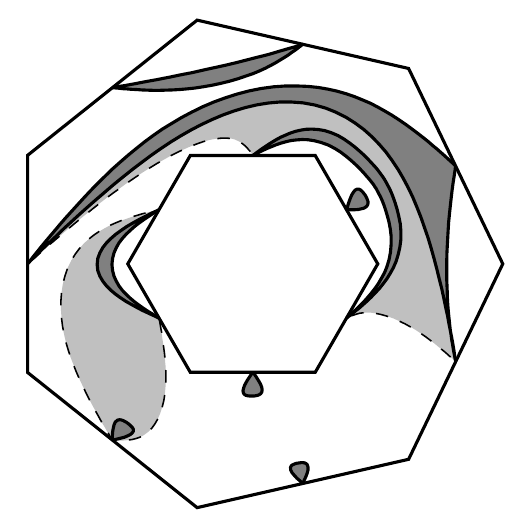&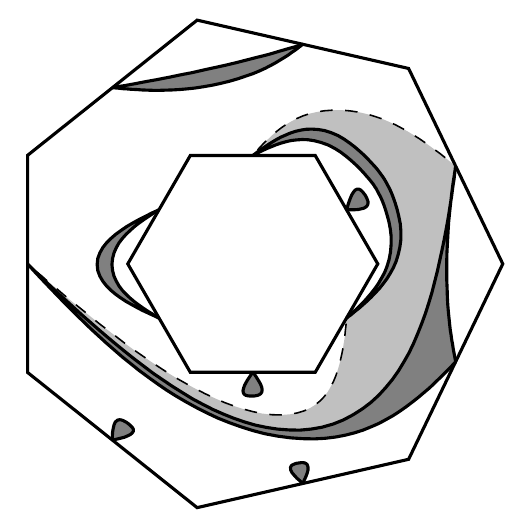\\
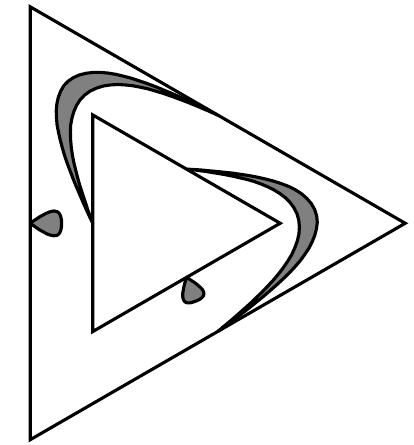&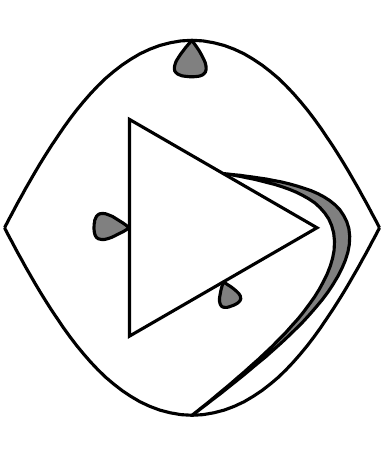\\
\end{tabular}
\caption{Two annular noncrossing permutations \(\pi\) constructed from the same \(\pi_{0}\) with different outside faces (above), and their restrictions to the outside faces.  See Example~\ref{example: annuli}.}
\end{figure}

\begin{example}
In the two top diagrams of Figure~\ref{figure: outside faces}, the annulus is given by \(\tau=\tau_{6,7}\).  Both diagrams show (in dark grey) the permutation
\[\pi_{0}=\left(1\right)\left(2,6\right)\left(3,4\right)\left(5\right)\left(7,10,13\right)\left(8\right)\left(9\right)\left(11,12\right)\in S_{\mathrm{disc-nc}}\left(6,7\right)\textrm{.}\]
The complement is given by either
\[\mathrm{Kr}_{\tau}\left(\pi_{0}\right)=\pi_{0}^{-1}\tau=\left(1,6\right)\left(2,4,5\right)\left(3\right)\left(7,8,9\right)\left(10,12\right)\left(11\right)\left(13\right)\]
or
\[\mathrm{Kr}_{\tau}^{-1}\left(\pi_{0}\right)=\tau\pi_{0}^{-1}=\left(1,2\right)\left(3,5,6\right)\left(4\right)\left(7\right)\left(8,9,10\right)\left(11,13\right)\left(12\right)\]
depending on whether the white intervals between the numbers share their label with the number counterclockwise or the number clockwise.  (Note that the cycle structures of these two permutations are the same, since they are conjugate.)

In the diagram on the left, the outside faces are \(\left(2,4,5\right)\) and \(\left(7,8,9\right)\) (respectively \(\left(3,5,6\right)\) and \(\left(8,9,10\right)\)).  In the diagram on the right, the outside face are \(\left(2,4,5\right)\) and \(\left(10,12\right)\) (respectively \(\left(3,5,6\right)\) and \(\left(11,13\right)\)).

In the diagram on the left, the dark and light grey together form the diagram
\[\pi=\left(1\right)\left(2,10,13,7,6\right)\left(3,4,9\right)\left(5\right)\left(8\right)\left(11,12\right)\in S_{\mathrm{ann-nc}}\left(6,7\right)\textrm{.}\]
The bridges are formed by connecting one cycle of \(\pi_{0}\) from each cycle of \(\tau\).

We can connect any two cycles of \(\pi_{0}\) adjacent to the outside faces as long as it does not create crossings.  (If we are considering a fixed \(\rho\succeq\pi_{0}\), we connect only those contained in two of the cycles of \(\rho\).)

In the diagram on the right, the complement is
\[\mathrm{Kr}^{-1}\left(\pi\right)=\pi^{-1}\tau=\left(1,6\right)\left(2,9\right)\left(3\right)\left(4,5,7,8\right)\left(10,12\right)\left(11\right)\left(13\right)\]
with bridges \(\left(2,9\right)\) and \(\left(4,5,7,8\right)\).  The elements are exactly the elements of the two outside faces of \(\mathrm{Kr}^{-1}\left(\pi_{0}\right)\).

In the diagram on the right, the dark and light grey together form the diagram
\[\pi=\left(1\right)\left(2,13,7,10,6\right)\left(3,4\right)\left(5\right)\left(8\right)\left(9\right)\left(11,12\right)\textrm{.}\]

The lower two diagrams show the restrictions to the outside faces.  We can think of each edge as representing a block adjacent to an outside face containing that number, which may be connected to another such block on the other cycle of \(\tau\) in a noncrossing way.  The bridges of the complement  of the restriction are the same as those of the complement of the original permutation (in the case on the left \(\left(2,9\right)\left(4,5,7,8\right)\), and on the case on the right, \(\left(2,10,12,5,4\right)\)).  The cycles of the complement which do not appear do not depend on \(\pi\) but only on \(\pi_{0}\).  (We have been considering \(\rho=\tau\).  If \(\rho\) is another permutation, we would restrict to two cycles of \(\rho\).)
\label{example: annuli}
\end{example}

The following proposition gives the M\"{o}bius function:
\begin{proposition}
Here \(\tau=\tau_{p,q}\) and \(\mathrm{Kr}=\mathrm{Kr}_{\tau}\).

If \(\pi,\rho\in S_{\mathrm{nc-sd}}\left(p,q\right)\) with \(\pi\preceq\rho\), and it is not true that both \(\pi\in S_{\mathrm{disc-nc}}\left(p,q\right)\) and \(\rho\in \hat{S}_{\mathrm{disc-nc}}\left(p,q\right)\), then
\begin{equation}
\mu_{\mathrm{sd}}\left(\pi,\rho\right)=\prod_{U\in \Pi\left(\mathrm{Kr}_{\rho}\left(\pi\right)\right)}\left(-1\right)^{\left|U\right|-1}C_{\left|U\right|-1}\textrm{.}
\label{equation: nc-sd easy}
\end{equation}

Let \(\pi,\rho\in S_{\mathrm{disc-nc}}\left(p,q\right)\).  Then
\begin{multline}
\mu_{\mathrm{sd}}\left(\pi,\hat{\rho}\right)\\=\sum_{\substack{U_{1},U_{2}\in\Pi\left(\mathrm{Kr}_{\rho}\left(\pi\right)\right)\\U_{1}\subseteq\left[p\right],U_{2}\subseteq\left[p+1,p+q\right]}}\left(-1\right)^{\left|U_{1}\right|+\left|U_{2}\right|}\gamma_{\left|U_{1}\right|,\left|U_{2}\right|}\prod_{U\in\Pi\left(\mathrm{Kr}_{\rho}\left(\pi\right)\right)\setminus\left\{U_{1},U_{2}\right\}}\left(-1\right)^{\left|U\right|-1}C_{\left|U\right|-1}\\-\prod_{U\in \Pi\left(\mathrm{Kr}_{\rho}\left(\pi\right)\right)}\left(-1\right)^{\left|U\right|-1}C_{\left|U\right|-1}\textrm{.}
\label{equation: nc-sd hard}
\end{multline}
\label{proposition: self dual}
\end{proposition}
\begin{proof}
Since \(S_{\mathrm{disc-nc}}\left(p,q\right)\cup S_{\mathrm{ann-nc}}\left(p,q\right)\) is isomorphic to the poset discussed in Section~\ref{section: permutations}, the M\"{o}bius function is given by (\ref{equation: nc-sd easy}) on this part of the poset.  Likewise, since \(S_{\mathrm{ann-nc}}\left(p,q\right)\cup\hat{S}_{\mathrm{disc-nc}}\left(p,q\right)\) is dual to \(S_{\mathrm{disc-nc}}\left(p,q\right)\cup S_{\mathrm{ann-nc}}\left(p,q\right)\), for \(\pi,\rho\in S_{\mathrm{ann-nc}}\left(p,q\right)\cup\hat{S}_{\mathrm{disc-nc}}\left(p,q\right)\),
\[\mu_{\mathrm{sd}}\left(\pi,\rho\right)=\mu\left(\widehat{\mathrm{Kr}}^{-1}\left(\rho\right),\widehat{\mathrm{Kr}}^{-1}\left(\pi\right)\right)=\mu\left(\tau\pi^{-1},\tau\rho^{-1}\right)\textrm{.}\]
Since
\[\mathrm{Kr}_{\tau\rho^{-1}}\left(\tau\pi^{-1}\right)=\left(\tau\pi^{-1}\right)^{-1}\tau\rho^{-1}=\pi\rho^{-1}=\left(\mathrm{Kr}_{\rho}\left(\pi\right)\right)^{-1}\textrm{,}\]
which has the same cycle structure as \(\mathrm{Kr}_{\rho}\left(\pi\right)\), the M\"{o}bius function is as given in (\ref{equation: nc-sd easy}).

Let \(\pi,\rho\in S_{\mathrm{disc-nc}}\left(p,q\right)\).  We wish to calculate
\[\mu_{\mathrm{sd}}\left(\pi,\hat{\rho}\right)=-\sum_{\sigma\in\left(\pi,\hat{\rho}\right]}\mu\left(\sigma,\hat{\rho}\right)\textrm{.}\]
We show first that the sum of the contribution from \(\sigma\) such that \(\sigma\succeq\upsilon\succ\pi\) for some \(\upsilon\in S_{\mathrm{disc-nc}}\left(p,q\right)\) vanishes, using the Principle of Inclusion and Exclusion (see, e.g.\ \cite{MR1311922}, Chapter 5).  This set is the union of the sets \(\left[\upsilon,\hat{\rho}\right]\) for \(\upsilon\in S_{\mathrm{disc-nc}}\left(p,q\right)\), \(\upsilon\succ\pi\).  Because \(S_{\mathrm{disc-nc}}\left(p,q\right)\) is a lattice, any \(\upsilon_{1},\ldots,\upsilon_{n}\in S_{\mathrm{disc-nc}}\left(p,q\right)\) have a least upper bound \(\upsilon^{\prime}\), so the intersection \(\left[\upsilon_{1},\hat{\rho}\right]\cap\cdots\cap\left[\upsilon_{n},\hat{\rho}\right]=\left[\upsilon^{\prime},\hat{\rho}\right]\).  Thus any intersection of the \(\left[\upsilon,\hat{\rho}\right]\) will be of this form.  Since \(\sum_{\sigma\in\left[\upsilon^{\prime},\hat{\rho}\right]}\mu_{\mathrm{sd}}\left(\sigma,\hat{\rho}\right)=0\), each set of multiply counted terms vanishes, and hence the sum of \(\mu_{\mathrm{sd}}\left(\sigma,\hat{\rho}\right)\) for \(\sigma\succeq\upsilon\) for any \(\upsilon\in S_{\mathrm{disc-nc}}\left(p,q\right)\) with \(\upsilon\succ\pi\) vanishes.

Thus, we need only consider the contribution of elements in \(\left(\pi,\hat{\rho}\right]\) not larger than any \(\upsilon\in S_{\mathrm{disc-nc}}\left(p,q\right)\) with \(\upsilon\succ\pi\); i.e., \(\hat{\pi}\) and \(\sigma\in S_{\mathrm{ann-nc}}\left(p,q\right)\) where \(\left.\sigma\right|_{\tau}=\pi\) and \(\sigma\prec\hat{\rho}\).  The value of \(\mu\left(\hat{\pi},\hat{\rho}\right)\) is calculated above, accounting for the last term in (\ref{equation: nc-sd hard}).

By Lemma~\ref{lemma: technical}, part~\ref{item: selected cycles}, the \(\sigma\in S_{\mathrm{ann-nc}}\left(p,q\right)\) with \(\pi\preceq\sigma\preceq\hat{\rho}\) and \(\left.\sigma\right|_{\tau}=\pi\) can be partitioned by which cycles of \(\mathrm{Kr}_{\rho}^{-1}\left(\pi\right)\) are the outside faces.  We sum over choices of blocks \(U_{1},U_{2}\in\Pi\left(\mathrm{Kr}_{\rho}^{-1}\left(\pi\right)\right)\) (\(U_{1}\in\left[p\right]\), \(U_{2}\in\left[p+1,p+q\right]\)).  By Lemma~\ref{lemma: technical}, part~\ref{item: spare cycles}, the cycles of \(\mathrm{Kr}_{\rho}^{-1}\left(\sigma\right)\) which are not contained in \(U_{1}\cup U_{2}\) are the same as those of \(\mathrm{Kr}_{\rho}^{-1}\left(\pi\right)\) (contributing the terms in the product in the first term of (\ref{equation: nc-sd hard})).  By part~\ref{item: skeleton}, the \(\sigma\) for this choice of \(U_{1}\) and \(U_{2}\) correspond to choices of noncrossing permutations on \(\left.\rho\right|_{U_{1}\cup U_{2}}\), all of whose cycles are bridges, which form the remaining cycles of \(\mathrm{Kr}_{\rho}^{-1}\left(\pi\right)\).  We compute the sum of the contribution of such permutations.

Given \(\sigma\in S_{\mathrm{ann-nc}}\left(r,s\right)\) a collection of bridges, we may construct a collection of bridges in \(S_{\mathrm{ann-nc}}\left(r+r^{\prime},s+s^{\prime}\right)\) by adding another bridge (with \(r^{\prime}\) elements in one end and \(s^{\prime}\) elements in the other) between the bridge containing \(1\) and the next bridge along that cycle of \(\tau_{r,s}\) (the one with the next smallest numbers).  This determines where the other end of this bridge lies in the other cycle of \(\tau_{r,s}\).  If we renumber the points so that the order is preserved, the permutation on \(S_{\mathrm{ann-nc}}\left(r+r^{\prime},s+s^{\prime}\right)\) is uniquely determined, unless the other end of the bridge falls between \(r+s\) and \(r+1\) (i.e.\ if \(1\) and \(r+1\) are in the same cycle, and \(\sigma\left(r+s\right)\neq r+1\), see Figure~\ref{figure: recurrence}, upper left).  In this case there are \(s^{\prime}+1\) ways to choose how the \(s^{\prime}\) elements of that end of the bridge are divided between the beginning and end of the interval \(\left[r+r^{\prime}+1,r+r^{\prime}+s+s^{\prime}\right]\).  If the permutation is numbered so the added bridge ends at \(r+r^{\prime}+s+s^{\prime}\), the permutation will still have the property that the bridge containing \(1\) also contains \(r+r^{\prime}+1\) and \(\sigma\left(r+r^{\prime}+s+s^{\prime}\right)\neq r+r^{\prime}+1\) (Figure~\ref{figure: recurrence}, upper right); otherwise the cycle containing \(1\) will not contain \(r+r^{\prime}+1\) (Figure~\ref{figure: recurrence}, lower left).  It is thus useful to consider these cases separately.  We also note that any permutation with at least two bridges is constructed in this manner from a unique permutation and \(r,s,r^{\prime},s^{\prime}\), since the process can be reversed by removing the first bridge after the one containing \(1\) and renumbering.

We define a generating function \(f_{1}\left(x,y\right)=\sum_{r,s\geq 1}f^{\left(1\right)}_{r,s}x^{r}y^{s}\) where the coefficient of \(x^{r}y^{s}\) is the sum of contributions of collections of bridges \(\sigma\in S_{\mathrm{ann-nc}}\left(r,s\right)\) where \(1\) is in the same cycle as \(r+1\) but \(\sigma\left(r+s\right)\neq r+1\):
\[f^{\left(1\right)}_{r,s}=\sum_{\substack{\sigma\in S_{\mathrm{ann-nc}}\left(r,s\right)\\\textrm{\(\sigma\) all bridges}\\\left.\sigma\right|_{\left\{1,r+1\right\}}=\left(1,r+1\right)\\\sigma\left(r+s\right)\neq r+1}}\prod_{U\in\Pi\left(\sigma\right)}\left(-1\right)^{\left|U\right|-1}C_{\left|U\right|-1}\textrm{.}\]
We let \(g_{1}\left(x,y\right)\) be the generating function of the added bridge (denoting the generating function of the Catalan numbers by \(C\left(x\right)=\frac{1-\sqrt{1-4x}}{2x}\)):
\begin{multline*}
g_{1}\left(x,y\right)=\sum_{r,s\geq 1}\left(-1\right)^{r+s-1}C_{r+s-1}x^{r}y^{s}\\=\sum_{k\geq 2}\left(-1\right)^{k-1}C_{k-1}\left(x^{k-1}y+x^{k-2}y^{2}+\cdots+x^{2}y^{k-2}+xy^{k-1}\right)\\=\sum_{k\geq 2}\left(-1\right)^{k-1}C_{k-1}xy\frac{x^{k-1}-y^{k-1}}{x-y}=xy\frac{\left(C\left(-x\right)-1\right)-\left(C\left(-y\right)-1\right)}{x-y}\\=\frac{1}{2}+\frac{y\sqrt{1+4x}-x\sqrt{1+4y}}{2\left(x-y\right)}\textrm{.}
\end{multline*}
We let \(h_{1}\left(x,y\right)=\sum_{x,y\geq 1}h^{\left(1\right)}x^{r}y^{s}\) be the generating function of the contribution of the permutations \(\sigma\) with only one bridge, where \(\sigma\left(r+s\right)\neq r+1\).  There are \(r\) choices for the first element \(a\) of the bridge in \(\left[r\right]\) and only one possible choice of first and last elements of the bridge in \(\left[r+1,r+s\right]\), so
\begin{multline*}
h_{1}\left(x,y\right)=\sum_{r,s\geq 1}\left(-1\right)^{r+s}rC_{r+s-1}x^{r}y^{s}=x\frac{\partial}{\partial x}g_{1}\left(x,y\right)\\=-\frac{xy\left(1+2x+2y-\sqrt{1+4x}\sqrt{1+4y}\right)}{2\sqrt{1+4x}\left(x-y\right)^{2}}\textrm{.}
\end{multline*}
(Here and throughout the generating functions may be derived by standard algebraic calculation, and we will skip many intermediate steps.)  Then \(f_{1}\) satisfies the recurrence relation
\[f_{1}=f_{1}g_{1}+h_{1}\]
so
\[f_{1}\left(x,y\right)=\frac{h_{1}\left(x,y\right)}{1-g_{1}\left(x,y\right)}=\frac{xy\left(1+2x+2y-\sqrt{1+4x}\sqrt{1+4y}\right)}{\sqrt{1+4x}\left(x-y\right)\left(y+y\sqrt{1+4x}-x-x\sqrt{1+4y}\right)}\textrm{.}\]

We define \(f_{2}\left(x,y\right)=\sum_{r,s\geq 1}f^{\left(2\right)}_{r,s}x^{r}y^{s}\) so \(f^{\left(2\right)}_{r,s}\) is the contribution of diagrams not considered in \(f_{2}\), i.e.\ where \(1\) is not in the same bridge as \(r+1\) or \(\sigma\left(r+s\right)=r+1\):
\[f^{\left(2\right)}_{r,s}=\sum_{\substack{\sigma\in S_{\mathrm{ann-nc}}\left(r,s\right)\\\textrm{\(\sigma\) all bridges}\\\textrm{\(\left.\sigma\right|_{\left\{1,r+1\right\}}=\left(1\right)\left(r+1\right)\) or \(\sigma\left(r+s\right)=r+1\)}}}\prod_{U\in\Pi\left(\sigma\right)}\left(-1\right)^{\left|U\right|-1}C_{\left|U\right|-1}\textrm{.}\]
We let \(g_{2}\left(x,y\right)=\sum_{r,s\geq 1}g^{\left(2\right)}_{r,s}x^{r}y^{s}\) be the generating function for the contribution of bridges added to a \(\sigma\) with \(1\) and \(r+1\) in the same bridge and \(\sigma\left(r+s\right)=r+1\) such that in the new \(\sigma^{\prime}\), \(r+r^{\prime}+s+s^{\prime}\) is no longer in the same cycle as \(1\).  There are \(s^{\prime}\) ways to add the new bridge, so:
\begin{multline*}
g_{2}\left(x,y\right)=\sum_{r,s\geq 1}\left(-1\right)^{r+s-1}sC_{r+s-1}x^{r}y^{s}=y\frac{\partial}{\partial y}g_{1}\left(x,y\right)\\=-\frac{xy\left(1+2x+2y-\sqrt{1+4x}\sqrt{1+4y}\right)}{2\sqrt{1+4y}\left(x-y\right)^{2}}\textrm{.}
\end{multline*}
Let \(h_{2}=\sum_{r,s\geq 1}h^{\left(2\right)}_{r,s}x^{r}y^{s}\) be the generating function of the \(\sigma\in S_{\mathrm{ann-nc}}\left(r,s\right)\) with only one bridge and \(\sigma\left(r+s\right)=r+1\).  There are \(r\) choices for the first element from \(\left[p\right]\) and \(s-1\) choices for the first element from \(\left[p+1,p+q\right]\) which satisfy this condition, so:
\begin{multline*}
h_{2}\left(x,y\right)=y\frac{\partial}{\partial y}h_{1}\left(x,y\right)-h_{1}\left(x,y\right)\\=\frac{xy^{2}\left(\left(1+x+3y\right)\sqrt{1+4x}-\left(1+3x+y\right)\sqrt{1+4y}\right)}{\sqrt{1+4x}\sqrt{1+4y}\left(x-y\right)^{3}}\textrm{.}
\end{multline*}
Then \(f_{2}\) satisfies the recurrence relation:
\[f_{2}=f_{1}g_{2}+f_{2}g_{1}+h_{2}\]
so
\begin{multline*}
f_{2}\left(x,y\right)=\frac{f_{1}\left(x,y\right)g_{2}\left(x,y\right)+h_{2}\left(x,y\right)}{1-g_{1}\left(x,y\right)}\\=\frac{x^{2}-2xy+2x^{2}y-6xy^{2}-\left(x-y\right)^{2}\sqrt{1+4y}+y^{2}\sqrt{1+4x}\sqrt{1+4y}}{2\sqrt{1+4x}\sqrt{1+4y}\left(x-y\right)^{2}}\textrm{.}
\end{multline*}
We let \(f\left(x,y\right)=\sum_{r,s\geq 1}f_{r,s}x^{r}y^{s}\) whose coefficients are the sum of the contribution of all the collections of bridges:
\[f_{r,s}=-\sum_{\substack{\sigma\in S_{\mathrm{ann-nc}}\left(r,s\right)\\\textrm{\(\sigma\) all bridges}}}\prod_{U\in\Pi\left(\sigma\right)}\left(-1\right)^{\left|U\right|-1}C_{\left|U\right|-1}\]
(negative since this is the value that will contribute to \(\mu\left(\pi,\hat{\rho}\right)\).  Then
\[f\left(x,y\right)=-f_{1}\left(x,y\right)-f_{2}\left(x,y\right)=\frac{xy\left(1+2x+2y-\sqrt{1+4x}\sqrt{1+4y}\right)}{2\sqrt{1+4x}\sqrt{1+4y}\left(x-y\right)^{2}}\textrm{.}\]
Since
\[\left.\frac{\partial}{\partial z}f\left(xz,yz\right)\right|_{z=1}=2\frac{x}{\left(1+4x\right)^{3/2}}\frac{y}{\left(1+4y\right)^{3/2}}\]
and
\begin{multline*}
\frac{x}{\left(1+4x\right)^{3/2}}=\sum_{k\geq 1}\left(-1\right)^{k-1}\frac{1\cdot 3\cdot 5\cdot\cdots\cdot\left(2k-1\right)}{2^{k-1}\left(k-1\right)!}\left(4x\right)^{k}\\=\sum_{k\geq 1}\left(-1\right)^{k-1}\frac{\left(2k-1\right)!}{\left(k-1\right)!^{2}}x^{k}
\end{multline*}
the coefficient \(f_{r,s}\) of \(x^{r}y^{s}\) in \(f\) is
\[\left(-1\right)^{p+q}\frac{2}{p+q}\frac{\left(2p-1\right)!}{\left(p-1\right)!^{2}}\frac{\left(2q-1\right)!}{\left(q-1\right)!^{2}}\textrm{.}\]
The expression for \(\mu\left(\pi,\hat{\rho}\right)\) follows.
\end{proof}

\begin{figure}
\centering
\begin{tabular}{cc}
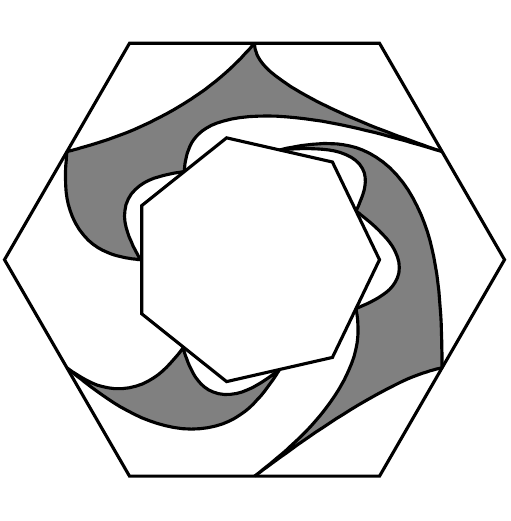&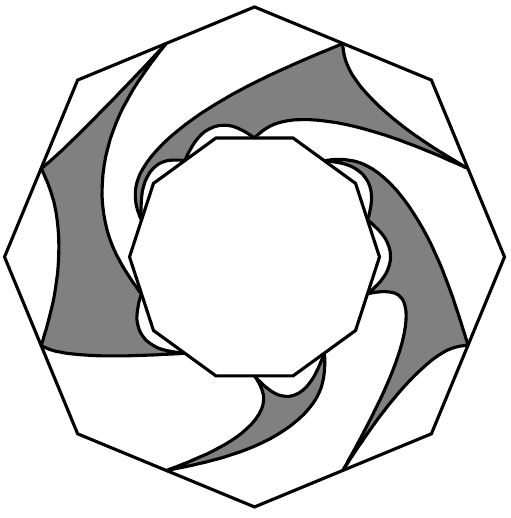\\
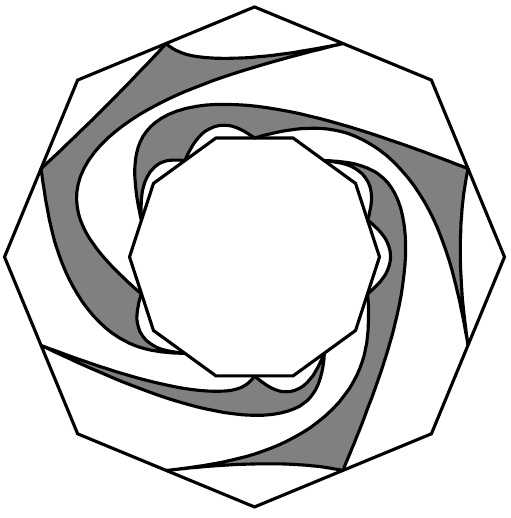&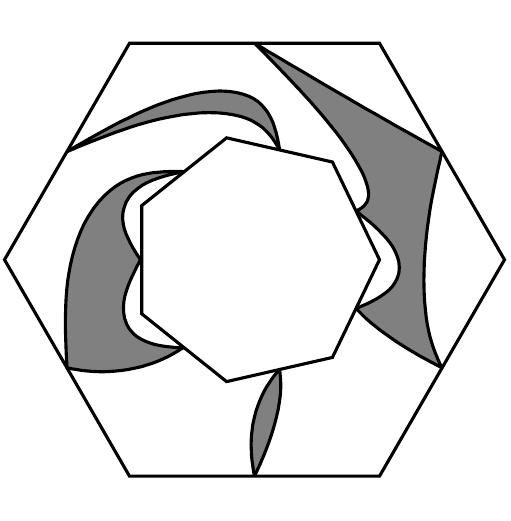\\
\end{tabular}
\caption{Top left: A diagram which contributes \(14\cdot 14\cdot 2x^{7}y^{6}=384x^{7}y^{6}\) to \(f_{1}\).  Top right: A diagram constructed from top left by adding a bridge.  This diagram contributes to \(f_{1}\).  Bottom left: A diagram constructed from top left by adding a bridge.  This diagram contributes to \(f_{2}\).  Bottom right: A diagram which contributes to \(f_{2}\).}
\label{figure: recurrence}
\end{figure}

\section{Minimal Length Partitioned Permutations}

\label{section: partitioned permutations}

Partitioned permutations, ordered pairs \(\left({\cal U},\pi\right)\) of a partition \({\cal U}\) and a permutation \(\pi\) where \({\cal U}\succeq\Pi\left(\pi\right)\), appear in higher-order freeness (see \cite{MR2294222, MR2302524}).  In particular, the partitioned permutations of minimal length correspond to asymptotically surviving terms in computations of cumulants of traces.  We will not give the definition of length here, but the set of relevant partitioned permutations is given below.  The M\"{o}bius function of such partitioned permutations is discussed in \cite{ppMobius}.

\begin{definition}
We define the set \({\cal PS}^{\prime}\left(r,s\right)\subseteq{\cal P}\left(r+s\right)\times S_{nc}\left(r,s\right)\) to be those where \(\pi\in S_{\mathrm{nc}}\left(p,q\right)\), and either \({\cal U}=\Pi\left(\pi\right)\); or \(\pi\in S_{\mathrm{disc-nc}}\left(p,q\right)\) and \({\cal U}\) is constructed by replacing two blocks of \(\Pi\left(\pi\right)\), one contained in each of \(\left[p\right]\) and \(\left[p+1,p+q\right]\), with their union (this block will be referred to as the nontrivial block).

In this poset, an element with a nontrivial block is never smaller than an element without one.  For \(\left({\cal U},\pi\right),\left({\cal V},\rho\right)\in{\cal PS}^{\prime}\left(p,q\right)\) where it is not the case that \({\cal U}\) has a nontrivial block and \({\cal V}\) does not have one, we let \(\left({\cal U},\pi\right)\preceq\left({\cal V},\rho\right)\) if \({\cal U}\preceq{\cal V}\) and \(\pi\in S_{\mathrm{nc}}\left(\rho\right)\).

We will often abbreviate the notation for elements \(\left({\cal U},\pi\right)\in{\cal PS}^{\prime}\left(p,q\right)\).  We will use a permutation \(\pi\) to denote \(\left(\Pi\left(\pi\right),\pi\right)\), and a partition to denote \(\left({\cal U},\Pi^{-1}\left(\left.{\cal U}\right|_{\tau}\right)\right)\).

We denote the smallest element, \(\left(1\right)\cdots\left(p+q\right)\), by \(0\), and the largest element, \(\left[1,\ldots,p+q\right]\), by \(1\).

We denote the M\"{o}bius function on this poset by \(\mu_{{\cal PS}^{\prime}}\).
\end{definition}

\begin{proposition}
Let \(\left({\cal U},\pi\right),\left({\cal V},\rho\right)\in{\cal PS}^{\prime}\left(p,q\right)\) with \(\left({\cal U},\pi\right)\preceq\left({\cal V},\rho\right)\).  If it is not the case that both \({\cal U}\in S_{\mathrm{disc-nc}}\left(p,q\right)\) and \({\cal V}\neq\Pi\left(\rho\right)\), then
\begin{equation}
\mu_{{\cal PS}^{\prime}}\left(\left({\cal U},\pi\right),\left({\cal V},\rho\right)\right)=\prod_{U\in\Pi\left(\mathrm{Kr}_{\rho}\left(\pi\right)\right)}\left(-1\right)^{\left|U\right|-1}C_{\left|U\right|-1}\textrm{.}
\label{equation: PS easy}
\end{equation}
If \({\cal U}\in S_{\mathrm{disc-nc}}\left(p,q\right)\) and \({\cal V}\) has nontrivial block \(V_{0}\), then
\begin{multline}
\mu_{{\cal PS}^{\prime}}\left(\left({\cal U},\pi\right),\left({\cal V},\rho\right)\right)\\=\sum_{\substack{U_{1},U_{2}\in\Pi\left(\left.\mathrm{Kr}_{\rho}\left(\pi\right)\right|_{V_{0}}\right)\\U_{1}\subseteq\left[p\right],U_{2}\subseteq\left[p+1,p+q\right]}}\left(-1\right)^{\left|U_{1}\right|+\left|U_{2}\right|}\gamma_{\left|U_{1}\right|,\left|U_{2}\right|}\\\times\prod_{U\in\Pi\left(\mathrm{Kr}_{\rho}\left(\pi\right)\right)\setminus\left\{U_{1},U_{2}\right\}}\left(-1\right)^{\left|U\right|-1}C_{\left|U\right|-1}\\-\#\left(\left.\pi\right|_{V_{0}\cap\left[p\right]}\right)\#\left(\left.\pi\right|_{V_{0}\cap\left[p+1,p+q\right]}\right)\prod_{U\in\Pi\left(\mathrm{Kr}_{\rho}\left(\pi\right)\right)}\left(-1\right)^{\left|U\right|-1}C_{\left|U\right|-1}\textrm{.}
\label{equation: PS hard}
\end{multline}
\end{proposition}
\begin{proof}
Since the portion of \({\cal PS}^{\prime}\left(p,q\right)\) excluding the elements with a nontrivial block is isomorphic to the poset of Section~\ref{section: permutations}, here the M\"{o}bius function will be given by (\ref{equation: PS easy}).

The interval \(\left[\left({\cal U},\pi\right),\left({\cal V},\rho\right)\right]\) is the product poset over blocks \(V\in{\cal V}\) of the restrictions of the ordered pairs to \(V\) (Lemma~\ref{lemma: technical}, part~\ref{item: piecewise}).  Over any trivial block, the contribution will be given by the the M\"{o}bius function of the disc-noncrossing permutations given in (\ref{equation: disc noncrossing}).  So we may restrict our attention to the nontrivial block, and thus to the case where \(\left({\cal V},\rho\right)=1\).

If \(\pi\in S_{\mathrm{ann-nc}}\left(p,q\right)\), the interval \(\left[\pi,1\right]\) consists of \(\sigma\in S_{\mathrm{ann-nc}}\left(p,q\right)\) and \(\left({\cal W},\sigma\right)\) with nontrivial block.  In the latter case, these are in bijection with \(\sigma\in S_{\mathrm{disc-nc}}\left(p,q\right)\) with \(\hat{\sigma}\succeq\pi\), since the nontrivial block of \({\cal W}\) is chosen uniquely to contain the two cycles of \(\sigma\) containing the bridges of \(\pi\).  This interval is thus isomorphic to the interval \(\left[\pi,1\right]\subseteq S_{\mathrm{nc-sd}}\left(p,q\right)\), and hence the M\"{o}bius function is given by (\ref{equation: PS easy}).

If \(\pi\in S_{\mathrm{disc-nc}}\left(p,q\right)\), then we may use apply the Principle of Inclusion and Exclusion as in Proposition~\ref{proposition: self dual}.  We may again ignore the contribution of any \(\left({\cal W},\sigma\right)\succeq\upsilon\succ\pi\) for any \(\upsilon\in S_{\mathrm{disc-nc}}\left(p,q\right)\).  The remaining elements of \({\cal PS}^{\prime}\left(p,q\right)\) are the \(\left({\cal W},\pi\right)\) with a nontrivial block and the \(\sigma\in S_{\mathrm{ann-nc}}\left(p,q\right)\) with \(\left.\sigma\right|_{\tau}=\pi\).  The contribution of the latter is as in (\ref{equation: nc-sd hard}).  For any \(\left({\cal W},\pi\right)\) with nontrivial block, \(\mu_{{\cal PS}^{\prime}}\left({\cal W},1\right)=\mu_{\mathrm{sd}}\left(\pi,1\right)\), and since the nontrivial block of \({\cal W}\) contains two blocks of \(\Pi\left(\pi\right)\), one contained in each of \(\left[p\right]\) and \(\left[p+1,p+q\right]\), there are \(\#\left(\left.\pi\right|_{\left[p\right]}\right)\#\left(\left.\pi\right|_{\left[p+1,p+q\right]}\right)\) such \({\cal W}\), hence the last term of (\ref{equation: PS hard}).
\end{proof}

\section{Annular Noncrossing Partitions}

\label{section: partitions}

Annular noncrossing partitions are discussed in \cite{MR1722987}.  The partition records which elements are in the same block, and it must be possible to represent the partition in a noncrossing way on the annulus.  However, distinct configurations as defined in previous sections may be represented by the same partition.  Specifically, any noncrossing partition with exactly one bridge may be represented by \(r\times s\) noncrossing permutations, where there are \(r\) elements of the bridge on one side of the bridge and \(s\) on the other side (see the Annular Nonstandard Condition~\ref{item: ans1} and Lemma~\ref{lemma: technical}, part~\ref{item: bridge}).  All other noncrossing partitions correspond to exactly one noncrossing permutation.

\begin{definition}
We let \({\cal P}_{\mathrm{nc}}\left(p,q\right)\) be the set of \(\Pi\left(\pi\right)\) for \(\pi\in S_{\mathrm{nc}}\left(p,q\right)\).

For \({\cal U},{\cal V}\in{\cal P}_{\mathrm{nc}}\left(p,q\right)\), we say \({\cal U}\preceq{\cal V}\) if this is the case as partitions.

We denote the smallest element, \(\left\{\left\{1\right\},\ldots,\left\{p+q\right\}\right\}\), by \(0\), and the largest element, \(\left\{\left[p+q\right]\right\}\), by \(1\).

We denote the M\"{o}bius function on this poset by \(\mu_{{\cal P}_{\mathrm{nc}}}\).
\end{definition}

This definition means that configurations which are not comparable as elements of \(S_{\mathrm{nc}}\left(p,q\right)\) are comparable in \({\cal P}_{\mathrm{nc}}\left(p,q\right)\), as shown in the following example.

\begin{example}
The partition \({\cal U}=\left\{\left\{1,5\right\},\left\{2,6\right\},\left\{3\right\},\left\{4\right\}\right\}\) (top left of Figure~\ref{figure: partitions}) is smaller in this poset than the partition \({\cal V}=\left\{\left\{1,2,5,6\right\},\left\{3,4\right\}\right\}\) (top right of Figure~\ref{figure: partitions}).  This is a necessary consequence of \({\cal U}\prec{\cal W}\prec{\cal V}\), where \({\cal W}=\left\{\left\{1,2,5,6\right\},\left\{3\right\},\left\{4\right\}\right\}\) (bottom left and right of Figure~\ref{figure: partitions}, which are represented by the same partition).
\label{example: partitions}
\end{example}

\begin{figure}
\centering
\begin{tabular}{cc}
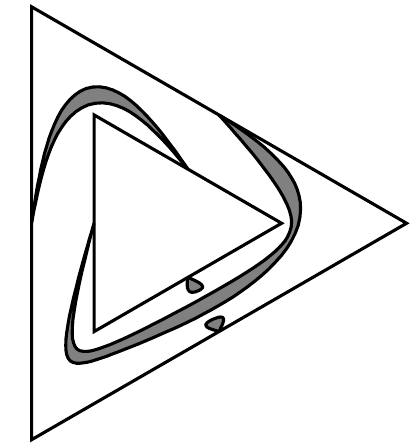&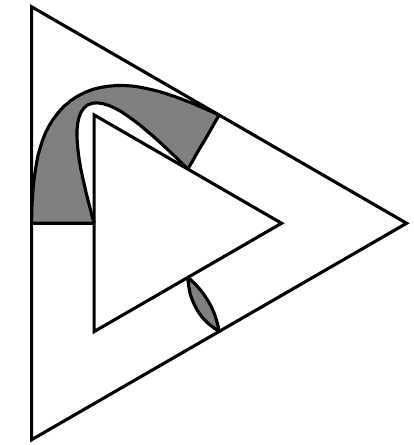\\
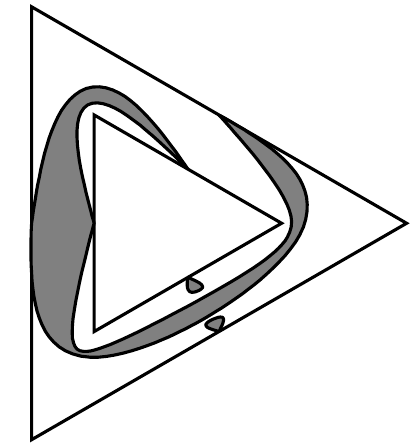&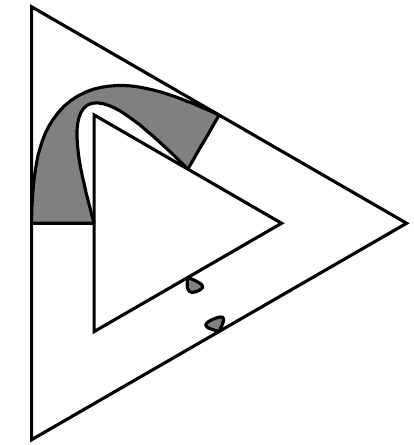\\
\end{tabular}
\caption{The partitions described in Example~\ref{example: partitions}.  Top left: \({\cal U}\).  Top right: \({\cal V}\).  Bottom left and right: two embeddings of the partition \({\cal W}\).}
\label{figure: partitions}
\end{figure}

In the following lemma we give an expression for a quantity which appears in the expression for the M\"{o}bius function for the noncrossing partitions.  The quantity may be viewed as the sum of \(\left(-1\right)^{k-1}C_{k-1}\times\left(-1\right)^{l-1}C_{l-1}\), where \(k\) and \(l\) are the number of elements respectively of two bridges with a total of \(p\) elements on one side and a total of \(q\) elements on the other, where one bridge is assumed to occupy the first segment of both \(\left[p\right]\) and \(\left[p+1,p+q\right]\), and the other the second segment.
\begin{lemma}
The following identity holds:
\begin{multline}
\sum_{i=1}^{p-1}\sum_{j=1}^{q-1}\left[\left(-1\right)^{i+q-j-1}C_{i+q-j-1}\right]\left[\left(-1\right)^{p-i+j-1}C_{p-i+j-1}\right]\\=\frac{\left(-1\right)^{p+q}}{\left(p+q\right)\left(p+q-1\right)}\frac{\left(2p\right)!}{p!\left(p-1\right)!}\frac{\left(2q\right)!}{q!\left(q-1\right)!}+\left(-1\right)^{p+q-1}C_{p+q-1}\\=\left(-1\right)^{p+q}\frac{2}{\left(p+q-1\right)}\gamma_{p,q}+\left(-1\right)^{p+q-1}C_{p+q-1}
\label{equation: two bridges}
\end{multline}
and thus
\begin{multline}
\sum_{i=1}^{p}\sum_{j=1}^{q}\left[\left(-1\right)^{i+q-j-1}C_{i+q-j-1}\right]\left[\left(-1\right)^{p-i+j-1}C_{p-i+j-1}\right]\\=\frac{\left(-1\right)^{p+q}}{\left(p+q\right)\left(p+q-1\right)}\frac{\left(2p\right)!}{p!\left(p-1\right)!}\frac{\left(2q\right)!}{q!\left(q-1\right)!}=\left(-1\right)^{p+q}\frac{2}{\left(p+q-1\right)}\gamma_{p,q}\textrm{.}
\label{equation: partition face}
\end{multline}
\end{lemma}
\begin{proof}
The generating function for the contribution of a bridge with \(i\) elements on one side and \(j\) elements on the other, \(i,j>0\), is \(g_{1}\) from the proof of Proposition~\ref{proposition: self dual}:
\[g_{1}\left(x,y\right)=\frac{1}{2}+\frac{y\sqrt{1+4x}-x\sqrt{1+4y}}{2\left(x-y\right)}\textrm{.}\]
Thus the generating function for (\ref{equation: partition face}) is
\[\sum_{p,q=1}^{\infty}\sum_{i=1}^{p-1}\sum_{j=1}^{q-1}\left(-1\right)^{k+l}C_{i+j-1}C_{p-i+q-j-1}x^{p}y^{q}=g_{1}\left(x,y\right)^{2}\textrm{.}\]
Since
\begin{multline*}
g_{1}\left(x,y\right)^{2}-\sum_{k=2}^{\infty}\left(-1\right)^{k-1}C_{k-1}\left(x^{k-1}y+\cdots+xy^{k-1}\right)=g_{1}\left(x,y\right)^{2}-g_{1}\left(x,y\right)\\=xy\frac{1+2x+2y-\sqrt{1+4x}\sqrt{1+4y}}{2\left(x-y\right)^{2}}
\end{multline*}
and
\[\left.\frac{\partial^{2}}{\partial z^{2}}\left(g_{1}\left(xz,yz\right)^{2}-g_{1}\left(xz,yz\right)\right)\right|_{z=1}=\frac{2xy}{\left(1+4x\right)^{3/2}\left(1+4y\right)^{3/2}}\]
we have (\ref{equation: two bridges}) (see the end of Proposition~\ref{proposition: self dual} for more detail).

The expression (\ref{equation: partition face}) is equal to the expression (\ref{equation: two bridges}) plus the additional terms corresponding to \(i=p\) or \(j=q\):
\begin{multline*}
\left(-1\right)^{p+q}\left[\sum_{j=1}^{q-1}C_{p+q-j-1}C_{j-1}+\sum_{i=1}^{p-1}C_{i-1}C_{p+q-i-1}+C_{p-1}C_{q-1}\right]\\=\left(-1\right)^{p+q}\sum_{k=0}^{}C_{k}C_{p+q-2-k}=\left(-1\right)^{p+q}C_{p+q-1}
\end{multline*}
(by a standard recursion formula for Catalan numbers), whence (\ref{equation: partition face}).
\end{proof}

\begin{proposition}
Let \({\cal U},{\cal V}\in{\cal P}_{\mathrm{nc}}\left(p,q\right)\).

If there is exactly one permutation \(\rho\in\Pi^{-1}\left({\cal V}\right)\) (i.e.\ if \({\cal V}\) has no bridges or more than one bridge) and there is a \(\pi\in\Pi^{-1}\left({\cal U}\right)\) such that \(\pi\preceq\rho\), then
\begin{equation}
\mu_{{\cal P}_{\mathrm{nc}}}\left({\cal U},{\cal V}\right)=\mu\left(\pi,\rho\right)\textrm{.}
\label{equation: P easy}
\end{equation}

If there is exactly one permutation \(\rho\in\Pi^{-1}\left({\cal V}\right)\) and there is no \(\pi\in\Pi^{-1}\left({\cal U}\right)\) with \(\pi\preceq\rho\), then
\begin{equation}
\mu_{{\cal P}_{\mathrm{nc}}}\left({\cal U},{\cal V}\right)=0\textrm{.}
\label{equation: P easiest}
\end{equation}

If \({\cal V}\) has exactly one bridge \(V_{0}\) and \({\cal U}\) has no bridges, then letting \(\rho_{0}=\Pi^{-1}\left({\cal V}\right)\),
\begin{multline}
\mu_{{\cal P}_{\mathrm{nc}}}\left({\cal U},{\cal V}\right)\\=\sum_{\substack{U_{1},U_{2}\in\Pi\left(\mathrm{Kr}_{\rho_{0}}\left({\cal U}\right)\right)\\U_{1}\subseteq\left[p\right],U_{2}\in\left[p+1,p+q\right]}}\left(-1\right)^{\left|U_{1}\right|+\left|U_{2}\right|}\left(\gamma_{\left|U_{1}\right|,\left|U_{2}\right|}-\left|U_{1}\right|\left|U_{2}\right|C_{\left|U_{1}\right|+\left|U_{2}\right|-1}\right)\\\times\prod_{U\in\Pi\left(\mathrm{Kr}_{\rho_{0}}\left({\cal U}\right)\right)\setminus\left\{U_{1},U_{2}\right\}}\left(-1\right)^{\left|U\right|-1}C_{\left|U\right|-1}\\-\#\left(\left.\pi\right|_{V_{0}\cap\left[p\right]}\right)\#\left(\left.\pi\right|_{V_{0}\cap\left[p+1,p+q\right]}\right)\prod_{U\in\Pi\left(\mathrm{Kr}_{\rho_{0}}\left(\pi\right)\right)}\left(-1\right)^{\left|U\right|-1}C_{\left|U\right|-1}\textrm{.}
\label{equation: P 0-1}
\end{multline}

If both \({\cal U}\) and \({\cal V}\) both have exactly one bridge then, denoting the bridge of \({\cal U}\) by \(U_{0}\), and letting \(\pi_{0}:=\Pi^{-1}\left({\cal U}\wedge\tau\right)\) and \(\rho_{0}:=\Pi^{-1}\left({\cal V}\wedge\tau\right)\):
\begin{multline}
\mu_{{\cal P}_{\mathrm{nc}}}\left({\cal U},{\cal V}\right)\\=\sum_{\substack{U_{1},U_{2}\in\Pi\left(\mathrm{Kr}_{\rho_{0}}\left(\pi_{0}\right)\right)\\U_{1}\in\left[p\right],U_{2}\in\left[p+1,p+q\right]\\\pi_{0}\left(U_{k}\right)\cup U_{0}\neq\emptyset}}\left(-1\right)^{\left|U_{1}\right|+\left|U_{2}\right|}\left[\frac{2}{\left|U_{1}\right|+\left|U_{2}\right|-1}\gamma_{\left|U_{1}\right|,\left|U_{2}\right|}-C_{\left|U_{1}\right|+\left|U_{2}\right|-1}\right]\\\times\prod_{U\in\Pi\left(\mathrm{Kr}_{\rho_{0}}\left(\pi_{0}\right)\right)\setminus\left\{U_{1},U_{2}\right\}}\left(-1\right)^{\left|U\right|-1}C_{\left|U\right|-1}+\prod_{U\in\Pi\left(\mathrm{Kr}_{\rho_{0}}\left(\pi_{0}\right)\right)}\left(-1\right)^{\left|U\right|-1}C_{\left|U\right|-1}\textrm{.}
\label{equation: P 1-1}
\end{multline}

If \({\cal V}\) has exactly one bridge and \({\cal U}\) has more than one bridge, then, denoting \(\rho_{0}:={\cal V}\wedge\tau\):
\begin{multline}
\mu_{{\cal P}_{\mathrm{nc}}}\left({\cal U},{\cal V}\right)\\=\sum_{\textrm{\(U_{0}\) bridge of \(\mathrm{Kr}_{\rho_{0}}\left({\cal U}\right)\)}}\left(-1\right)^{\left|U_{0}\right|}\frac{2}{\left|U_{0}\right|-1}\gamma_{\left|U_{0}\cap\left[p\right]\right|,\left|U_{0}\cap\left[p+1,p+q\right]\right|}\\\times\prod_{U\in\Pi\left(\mathrm{Kr}_{\rho_{0}}\left({\cal U}\right)\right)\setminus\left\{U_{0}\right\}}\left(-1\right)^{\left|U\right|-1}C_{\left|U\right|-1}+\prod_{U\in\Pi\left(\mathrm{Kr}_{\rho_{0}}\left({\cal U}\right)\right)}\left(-1\right)^{\left|U\right|-1}C_{\left|U\right|-1}\textrm{.}
\label{equation: P n-1}
\end{multline}
\end{proposition}
\begin{proof}
We divide the proof into eight cases.
\begin{case}Neither \({\cal U}\) or \({\cal V}\) have bridges.

In this case \(\left[{\cal U},{\cal V}\right]\) is isomorphic to \(\left[\Pi^{-1}\left({\cal U}\right),\Pi^{-1}\left({\cal V}\right)\right]\), so \(\mu_{{\cal P}_{\mathrm{nc}}}\left({\cal U},{\cal V}\right)\) is given by (\ref{equation: P easy}).
\label{case: 0-0}
\end{case} 
\begin{case}\({\cal U}\) has more than one bridge, contained in more than one distinct bridge of \({\cal V}\).

Let \({\cal W}_{1},{\cal W}_{2}\in\left[{\cal U},{\cal V}\right]\) with \({\cal W}_{1}\preceq{\cal W}_{2}\).  Then the bridges of \({\cal W}_{1}\) are contained in more than one bridge of \({\cal W}_{2}\), since bridges of \({\cal U}\) contained in distinct bridges of \({\cal V}\) must be contained in bridges of \({\cal W}_{1}\) which must be contained in bridges of \({\cal W}_{2}\), which must be distinct since they are in distinct bridges of \({\cal V}\).

We note that for any such \({\cal W}_{1}\) and \({\cal W}_{2}\), \(\Pi^{-1}\left({\cal W}_{1}\right)\preceq\Pi^{-1}\left({\cal W}_{2}\right)\), since any elements which in \(\Pi^{-1}\left({\cal W}_{1}\right)\) satisfy the disc-nonstandard or disc-crossing conditions on a cycle of \(\Pi^{-1}\left({\cal W}_{2}\right)\) must satisfy the disc-nonstandard or disc-crossing conditions on \(\lambda_{x,y}\) with \(x\) and \(y\) chosen from a bridge of \({\cal W}_{1}\) not contained in that cycle of \(\Pi^{-1}\left({\cal W}_{2}\right)\), and hence must satisfy one of the annular-noncrossing conditions~\ref{item: ac2} or \ref{item: ac3}.

Thus \(\left[{\cal U},{\cal V}\right]\) is isomorphic to \(\left[\Pi^{-1}\left({\cal U}\right),\Pi^{-1}\left({\cal V}\right)\right]\), so the M\"{o}bius function is given by (\ref{equation: P easy}).
\label{case: n-n}
\end{case}
\begin{case}all bridges of \({\cal U}\) (there is at least one) are contained in one bridge of \({\cal V}\), \({\cal V}\) has more than one bridge, and there is no \(\pi\in\Pi^{-1}\left({\cal U}\right)\) such that \(\pi\preceq\Pi^{-1}\left({\cal V}\right)\).

We consider the set of \(\sigma\in S_{\mathrm{nc}}\left(p,q\right)\) with \(\sigma\preceq\Pi^{-1}\left({\cal V}\right)\) and such that \(\Pi\left(\sigma\right)\succeq{\cal U}\).  This set is a subset of \(S_{\mathrm{disc-nc}}\left(\Pi^{-1}\left({\cal V}\right)\right)\), which is a lattice.  If \(\sigma_{1},\sigma_{2}\) are members of this set, then \(\sigma_{1}\wedge\sigma_{2}\) exists and must also be a member of the set, since \(\Pi\left(\sigma_{1}\wedge\sigma_{2}\right)=\Pi\left(\sigma_{1}\right)\wedge\Pi\left(\sigma_{2}\right)\succeq{\cal U}\).  This set then contains its greatest lower bound, i.e.\ a \(\sigma_{0}\) such that \(\sigma_{0}\preceq\sigma\) for all \(\sigma\) in the set (the meet of all of the finitely many elements of the set).  We let \({\cal W}_{0}=\Pi\left(\sigma_{0}\right)\).  Since \(\sigma_{0}\preceq\Pi^{-1}\left({\cal V}\right)\), \({\cal W}_{0}\neq{\cal U}\), and since \({\cal W}_{0}\succeq{\cal U}\), \(\left[{\cal W}_{0},{\cal V}\right]\subseteq\left[{\cal U},{\cal V}\right]\).  We also note that \({\cal W}_{0}\neq{\cal V}\), since replacing a bridge (not containing the bridges of \({\cal U}\)) of \(\Pi^{-1}\left({\cal V}\right)\) with the cycles it induces on each of \(\left[p\right]\) and \(\left[p+1,p+q\right]\) creates a smaller permutation \(\sigma\) with \(\Pi\left(\sigma\right)\succeq{\cal U}\).

We consider the expression:
\begin{multline}
\mu_{{\cal P}_{\mathrm{nc}}}\left({\cal U},{\cal V}\right)=-\sum_{{\cal W}\in\left({\cal U},{\cal V}\right]}\mu_{{\cal P}_{\mathrm{nc}}}\left({\cal W},{\cal V}\right)\\=-\sum_{{\cal W}\in\left[{\cal W}_{0},{\cal V}\right]}\mu_{{\cal P}_{\mathrm{nc}}}\left({\cal W},{\cal V}\right)-\sum_{{\cal W}\in\left({\cal U},{\cal V}\right]\setminus\left[{\cal W}_{0},{\cal V}\right]}\mu_{{\cal P}_{\mathrm{nc}}}\left({\cal W},{\cal V}\right)\\=0-\sum_{{\cal W}\in\left({\cal U},{\cal V}\right]\setminus\left[{\cal W}_{0},{\cal V}\right]}\mu_{{\cal P}_{\mathrm{nc}}}\left({\cal W},{\cal V}\right)\textrm{.}
\label{equation: extraneous partitions}
\end{multline}
We proceed by induction on \(\#\left({\cal U}\right)-\#\left({\cal W}_{0}\right)\).

For the base case, if this value is minimal, then there can be no elements other than \({\cal U}\) in \(\left[{\cal U},{\cal V}\right]\setminus\left[{\cal W}_{0},{\cal V}\right]\) (since if \({\cal W}\in\left[{\cal U},{\cal V}\right]\) with a \(\sigma\in\Pi^{-1}\left({\cal W}\right)\) with \(\sigma\preceq\Pi^{-1}\left({\cal V}\right)\), it will be in \(\left[{\cal W}_{0},{\cal V}\right]\), and if there is no such \(\sigma\), then \(\#\left({\cal W}\right)-\#\left({\cal W}_{0}\right)<\#\left({\cal U}\right)-\#\left({\cal W}_{0}\right)\)).  Thus there are no terms in the sum on the right-hand side of (\ref{equation: extraneous partitions}), and hence the M\"{o}bius function is given by (\ref{equation: P easiest}).

For the induction step, any term in the sum on the right-hand side of (\ref{equation: extraneous partitions}) is covered by the induction step and thus vanishes, and hence the M\"{o}bius function is agan given by (\ref{equation: P easiest}).
\label{case: 1-n 0}
\end{case}
\begin{case}all bridges of \({\cal U}\) (there is at least one) are contained in one bridge of \({\cal V}\), \({\cal V}\) has more than one bridge, and there is an element \(\pi\in\Pi^{-1}\left({\cal U}\right)\) such that \(\pi\preceq\Pi^{-1}\left({\cal V}\right)\).

We proceed by induction on \(\#\left({\cal V}\right)-\#\left({\cal U}\right)\), noting that in the base case \(\#\left({\cal V}\right)-\#\left({\cal U}\right)=1\), \({\cal V}\) covers \({\cal U}\), so \(\mu_{{\cal P}_{\mathrm{nc}}}\left({\cal U},{\cal V}\right)=-1\), consistent with (\ref{equation: P easy}).

We construct \(\sigma_{0}\) from \(\Pi^{-1}\left({\cal V}\right)\) by splitting each bridge except the one containing the bridges of \(\Pi^{-1}\left({\cal U}\right)\) into two cycles: the cycle induced on \(\left[p\right]\) and the cycle induced on \(\left[p+1,p+q\right]\).  (It is easily verified that \(\sigma_{0}\) is noncrossing on \(\Pi^{-1}\left({\cal V}\right)\).)

If \({\cal W}\in\left[{\cal U},{\cal V}\right]\) does not have more than one bridge, that bridge must contain the bridges of \({\cal U}\) and hence must be contained in the remaining bridge of \(\Pi\left(\sigma_{0}\right)\), and thus \({\cal W}\preceq\Pi\left(\sigma_{0}\right)\).  Thus any \({\cal W}\in\left[{\cal U},{\cal V}\right)\setminus\left[{\cal U},\Pi\left(\sigma_{0}\right)\right]\) must have more than one bridge, so any nonvanishing \(\mu_{{\cal P}_{\mathrm{nc}}}\left({\cal U},{\cal W}\right)\) is given by (\ref{equation: P easy}) by either Case~\ref{case: n-n}, Case~\ref{case: 1-n 0}, or the induction hypothesis.  Then
\begin{multline*}
\mu_{{\cal P}_{\mathrm{nc}}}\left({\cal U},{\cal V}\right)=-\sum_{{\cal W}\in\left[{\cal U},{\cal V}\right)}\mu_{{\cal P}_{\mathrm{nc}}}\left({\cal U},{\cal W}\right)\\=-\sum_{{\cal W}\in\left[{\cal U},\Pi\left(\sigma_{0}\right)\right]}\mu_{{\cal P}_{\mathrm{nc}}}\left({\cal U},{\cal W}\right)-\sum_{{\cal W}\in\left[{\cal U},{\cal V}\right)\setminus\left[{\cal U},\Pi\left(\sigma_{0}\right)\right]}\mu_{{\cal P}_{\mathrm{nc}}}\left({\cal U},{\cal W}\right)
\end{multline*}
The first sum on the right-hand side vanishes, and is hence equal to the first sum on the right-hand side of the equation below.  Any nonzero term in the second sum corresponds to a \({\cal W}\) with exactly one \(\sigma\in\Pi^{-1}\left({\cal W}\right)\) in \(\left[\pi,\Pi^{-1}\left({\cal V}\right)\right)\setminus\left[\pi,\sigma_{0}\right]\) (as in Case~\ref{case: n-n}, \(\sigma\preceq\Pi^{-1}\left({\cal V}\right)\)).  Thus
\[\mu_{{\cal P}_{\mathrm{nc}}}\left({\cal U},{\cal V}\right)=-\sum_{\sigma\in\left[\pi,\sigma_{0}\right]}\mu\left(\pi,\sigma\right)-\sum_{\sigma\in\left[\pi,\Pi^{-1}\left({\cal V}\right)\right)\setminus\left[\pi,\sigma_{0}\right]}\mu\left(\pi,\sigma\right)=\mu\left(\pi,\Pi^{-1}\left({\cal V}\right)\right)\]
so the M\"{o}bius function is given by (\ref{equation: P easy}).
\label{case: 1-n}
\end{case}
\begin{case}\({\cal U}\) has no bridges; \({\cal V}\) has more than one bridge.

We proceed by induction on \(\#\left({\cal U}\right)-\#\left({\cal V}\right)\).  The smallest possible value in this case is \(2\) (when \({\cal V}\) has two bridges and \({\cal U}={\cal V}\wedge\tau\)), and by inspection there are two intermediate terms (with one bridge each), which cover \({\cal U}\) and are covered by \({\cal V}\).  Thus in this case (\ref{equation: P easy}) holds.

We construct \(\sigma_{1}\) by splitting all bridges of \(\Pi^{-1}\left({\cal V}\right)\) except one into the cycles it induces on \(\left[p\right]\) and \(\left[p+1,p+q\right]\).  We define \(\sigma_{2}\) similarly, choosing a different bridge of \(\Pi^{-1}\left({\cal V}\right)\).  We note that \(\Pi\left(\sigma_{1}\right)\wedge\Pi\left(\sigma_{2}\right)={\cal V}\wedge\tau\) and \(\sigma_{1}\wedge\sigma_{2}=\left.\Pi^{-1}\left({\cal V}\right)\right|_{\tau}\).  Thus (using the Principle of Inclusion and Exclusion):
\begin{multline*}
\mu_{{\cal P}_{\mathrm{nc}}}\left({\cal U},{\cal V}\right)=-\sum_{{\cal W}\in\left[{\cal U},{\cal V}\right)}\mu_{{\cal P}_{\mathrm{nc}}}\left({\cal U},{\cal W}\right)\\=-\sum_{{\cal W}\in\left[{\cal U},\Pi\left(\sigma_{1}\right)\right]}\mu_{{\cal P}_{\mathrm{nc}}}\left({\cal U},{\cal W}\right)-\sum_{{\cal W}\in\left[{\cal U},\Pi\left(\sigma_{2}\right)\right]}\mu_{{\cal P}_{\mathrm{nc}}}\left({\cal U},{\cal W}\right)+\sum_{{\cal W}\in\left[{\cal U},{\cal V}\wedge\tau\right]}\mu_{{\cal P}_{\mathrm{nc}}}\left({\cal U},{\cal W}\right)\\-\sum_{{\cal W}\in\left[{\cal U},{\cal V}\right)\setminus\left(\left[{\cal U},\Pi\left(\sigma_{1}\right)\right]\cup\left[{\cal U},\Pi\left(\sigma_{2}\right)\right]\right)}\mu_{{\cal P}_{\mathrm{nc}}}\left({\cal U},{\cal W}\right)\textrm{.}
\end{multline*}
The first two sums on the right-hand side vanish, and are hence equal to the first two sums on the first right-hand side of the following expression.  The terms in the third sum are given by Case~\ref{case: 0-0} (or the value of the sum by the properties of the M\"{o}bius function), so it is equal to the third sum on the right-hand side of the following expression.  In the fourth sum, all the \({\cal W}\) have at least two bridges, so the induction hypothesis applies to the M\"{o}bius function, and correspond to \(\sigma\in S_{\mathrm{nc}}\left(p,q\right)\) in the appropriate set (by the reasoning in Case~\ref{case: n-n}).  Thus:
\begin{multline*}
\mu_{{\cal P}_{\mathrm{nc}}}\left({\cal U},{\cal V}\right)=-\sum_{\sigma\in\left[{\cal U},\sigma_{1}\right]}\mu\left({\cal U},\sigma\right)-\sum_{\sigma\in\left[{\cal U},\sigma_{2}\right]}\mu\left({\cal U},\sigma\right)+\sum_{\sigma\in\left[{\cal U},\sigma_{1}\wedge\sigma_{2}\right]}\mu\left({\cal U},\sigma\right)\\-\sum_{\sigma\in\left[{\cal U},{\cal V}\right)\setminus\left(\left[{\cal U},\sigma_{1}\right]\cup\left[{\cal U},\sigma_{2}\right]\right)}\mu\left({\cal U},\sigma\right)=\sum_{\sigma\in\left[{\cal U},{\cal V}\right)}\mu\left({\cal U},\sigma\right)=\mu\left({\cal V},{\cal U}\right)
\end{multline*}
so the M\"{o}bius function is given by (\ref{equation: P easy}).
\end{case}
\begin{case}\({\cal U}\) has no bridges and \({\cal V}\) has one bridge.

We first show that
\begin{equation}
\mu_{{\cal P}_{\mathrm{nc}}}\left({\cal U},{\cal V}\right)=\sum_{\rho:\left({\cal V},\rho\right)\in{\cal PS}^{\prime}\left(p,q\right)}\mu_{{\cal PS}^{\prime}}\left({\cal U},\left({\cal V},\rho\right)\right)\textrm{.}
\label{equation: 0-1}
\end{equation}
Each term in the sum is defined, since \({\cal U}\preceq{\cal V}\wedge\tau\).  We proceed by induction on \(\#\left({\cal U}\right)-\#\left({\cal V}\right)\).

In the base case, \({\cal V}\) covers \({\cal U}\), and
\begin{multline*}
\sum_{\rho:\left({\cal V},\rho\right)\in{\cal PS}^{\prime}\left(p,q\right)}\mu_{{\cal PS}^{\prime}}\left({\cal U},\left({\cal V},\rho\right)\right)\\=\sum_{\left({\cal W},\sigma\right)\in\left[{\cal U},{\cal V}\right]}\mu_{{\cal PS}^{\prime}}\left({\cal U},\left({\cal W},\sigma\right)\right)-\mu_{{\cal PS}^{\prime}}\left({\cal U},{\cal U}\right)=0-1\textrm{.}
\end{multline*}

For the induction step,
\begin{multline*}
\mu_{{\cal P}_{\mathrm{nc}}}\left({\cal U},{\cal V}\right)=-\sum_{{\cal W}\in\left[{\cal U},{\cal V}\right)}\mu_{{\cal P}_{\mathrm{nc}}}\left({\cal U},{\cal W}\right)\\=-\sum_{\substack{{\cal W}\in\left[{\cal U},{\cal V}\right)\\\textrm{\({\cal W}\) one bridge}}}\mu_{{\cal P}_{\mathrm{nc}}}\left({\cal U},{\cal W}\right)-\sum_{\substack{{\cal W}\in\left[{\cal U},{\cal V}\right)\\\textrm{\({\cal W}\) not one bridge}}}\mu_{{\cal P}_{\mathrm{nc}}}\left({\cal U},{\cal W}\right)\\=-\sum_{\substack{\left({\cal W},\sigma\right)\in{\cal PS}^{\prime}\left(p,q\right)\\{\cal W}\in\left[{\cal U},{\cal V}\right)\\\textrm{\({\cal W}\) one bridge}}}\mu_{{\cal PS}^{\prime}}\left({\cal U},\left({\cal W},\sigma\right)\right)-\sum_{\substack{{\cal W}\in\left[{\cal U},{\cal V}\right)\\\textrm{\({\cal W}\) not one bridge}}}\mu_{{\cal PS}^{\prime}}\left({\cal U},{\cal W}\right)
\end{multline*}
The terms in the two sums are the \(\left({\cal W},\sigma\right)\in\left[{\cal U},{\cal V}\right]\) (\({\cal U}\) is noncrossing on any of the \(\sigma\) since \({\cal U}\preceq\sigma\wedge\tau\preceq\sigma\); and such a \(\sigma\) is noncrossing on \(\Pi^{-1}\left({\cal V}\wedge\tau\right)\) by Lemma~\ref{lemma: technical}, part~\ref{item: piecewise}).  All \(\left({\cal W},\sigma\right)\in\left[{\cal U},{\cal V}\right]\) appear in one of the two sums except those for which \({\cal W}={\cal V}\), i.e.\ those which appear in (\ref{equation: 0-1}).

The poset \(\left[{\cal U},{\cal V}\right]\) is the product of the posets restricted to each block of \({\cal V}\) (Lemma~\ref{lemma: technical}, part~\ref{item: piecewise}).  The contribution of each block of \({\cal V}\) that is not the bridge is given by (\ref{equation: disc noncrossing}), so we may restrict our attention to the case \({\cal V}=1\).

The \(\rho\) with \(\left(1,\rho\right)\) either have \(\Pi\left(\rho\right)=1\) or \(\rho=\tau\).  We consider these two cases.

The \(\rho\in\Pi^{-1}\left(1\right)\) are \(\left(a,\tau\left(a\right),\ldots,\tau^{p-1}\left(a\right),b,\tau\left(b\right),\ldots,\tau^{q-1}\left(b\right)\right)\) for some \(a\in\left[p\right]\) and \(b\in\left[p+1,p+q\right]\) (all other possibilities are ruled out by the annular nonstandard conditions~\ref{item: ans1} and \ref{item: ans2}), so \(\mathrm{Kr}^{-1}\left(\rho\right)=\left(a,b\right)\).  Letting \(\pi:=\mathrm{Kr}^{-1}\left({\cal U}\right)\), if \(a\in U_{1}\in\Pi\left(\pi\right)\) and \(b\in U_{2}\in\Pi\left(\pi\right)\)), by calculation the permutation \(\mathrm{Kr}_{\pi}\left(\mathrm{Kr}^{-1}\left(\rho\right)\right)\) contains the cycle
\[\left(a,\pi\left(a\right),\ldots,\pi^{-1}\left(a\right),b,\pi\left(b\right),\ldots,\pi^{-1}\left(b\right)\right)\]
while all other cycles are the cycles of \(\pi\).  Thus
\begin{multline*}
\mu\left(\Pi^{-1}\left({\cal U}\right),\rho\right)=\mu\left(\mathrm{Kr}^{-1}\left(\rho\right),\mathrm{Kr}^{-1}\left({\cal U}\right)\right)\\=\left(-1\right)^{\left|U_{1}\right|+\left|U_{2}\right|-1}C_{\left|U_{1}\right|+\left|U_{2}\right|-1}\prod_{U\in\Pi\left(\mathrm{Kr}^{-1}\left({\cal U}\right)\right)\setminus\left\{U_{1},U_{2}\right\}}\left(-1\right)^{\left|U\right|-1}C_{\left|U\right|-1}
\end{multline*}
so the sum over all such \(\rho\) is
\begin{multline*}
\sum_{\substack{U_{1},U_{2}\in\Pi\left(\mathrm{Kr}\left({\cal U}\right)\right)\\U_{1}\subseteq\left[p\right],U_{2}\in\left[p+1,p+q\right]}}\left(-1\right)^{\left|U_{1}\right|+\left|U_{2}\right|-1}\left|U_{1}\right|\left|U_{2}\right|C_{\left|U_{1}\right|+\left|U_{2}\right|-1}\\\times\prod_{U\in\Pi\left(\mathrm{Kr}^{-1}\left({\cal U}\right)\right)\setminus\left\{U_{1},U_{2}\right\}}\left(-1\right)^{\left|U\right|-1}C_{\left|U\right|-1}\textrm{.}
\end{multline*}

In the case where \(\rho=\tau\), the value of \(\mu_{{\cal PS}^{\prime}}\left({\cal U},{\cal V}\right)\) is given by (\ref{equation: PS hard}).  The value (\ref{equation: P 0-1}) follows.
\label{case: 0-1}
\end{case}

\begin{case}\({\cal U}\) and \({\cal V}\) each have exactly one bridge.

We will use the notation around (\ref{equation: P 1-1}).  We show first that
\begin{equation}
\mu_{{\cal P}_{\mathrm{nc}}}\left({\cal U},{\cal V}\right)=\sum_{\pi:\left({\cal U},\pi\right)\in{\cal PS}^{\prime}\left(p,q\right)}\sum_{\substack{\rho:\left({\cal V},\rho\right)\in{\cal PS}^{\prime}\left(p,q\right)\\\left({\cal V},\rho\right)\succeq\left({\cal U},\pi\right)}}\mu_{{\cal PS}^{\prime}}\left(\left({\cal U},\pi\right),\left({\cal V},\rho\right)\right)
\label{equation: 1-1}
\end{equation}
by induction on \(\#\left({\cal U}\right)-\#\left({\cal V}\right)\).

For the base case, \({\cal U}={\cal V}\).  Let the bridge of \({\cal U}\) have \(r\) elements on one side and \(s\) elements on the other.  The \(\left({\cal U},\pi\right)\) are those where \(\Pi\left(\pi\right)={\cal U}\) and \(\left({\cal U},\pi_{0}\right)\), i.e.\ the elements of \(\left(\pi_{0},{\cal U}\right]\subseteq{\cal PS}^{\prime}\left(p,q\right)\).  Thus the contribution to (\ref{equation: 1-1}) by terms where \(\left({\cal V},\rho\right)=\left({\cal U},\pi_{0}\right)\) is
\[\sum_{\left({\cal W},\sigma\right)\in\left(\pi_{0},{\cal U}\right]}\mu_{{\cal PS}^{\prime}}\left(\left({\cal W},\sigma\right),\left({\cal V},\rho\right)\right)=-\mu_{{\cal PS}^{\prime}}\left(\pi_{0},{\cal U}\right)=-rs+1\]
(by (\ref{equation: PS hard})).
The contribution of the remainder of the terms to \(\mu_{{\cal PS}^{\prime}}\left({\cal U},{\cal V}\right)\) are those where \(\left({\cal V},\rho\right)\) has \(\Pi\left(\rho\right)={\cal V}\); there are \(rs\) choices of such a \(\rho\).  For a given \(\left({\cal V},\rho\right)\), the only \(\left({\cal U},\pi\right)\) for which \(\left({\cal V},\rho\right)\succeq\left({\cal U},\pi\right)\) is when \(\pi=\rho\), with contribution \(\mu_{{\cal PS}^{\prime}}\left(\pi,\rho\right)=1\).  Thus the total contribution of such terms is \(rs\).  The total is then  \(1\), as desired.

For the induction step:
\begin{multline*}
\mu_{{\cal P}_{\mathrm{nc}}}\left({\cal U},{\cal V}\right)=-\sum_{{\cal W}\in\left[{\cal U},{\cal V}\right)}\mu_{{\cal P}_{\mathrm{nc}}}\left({\cal U},{\cal W}\right)\\=-\sum_{\substack{{\cal W}\in\left[{\cal U},{\cal V}\right)\\\textrm{\({\cal W}\) one bridge}}}\mu_{{\cal P}_{\mathrm{nc}}}\left({\cal U},{\cal W}\right)-\sum_{\substack{{\cal W}\in\left[{\cal U},{\cal V}\right)\\\textrm{\({\cal W}\) more than one bridge}}}\mu_{{\cal P}_{\mathrm{nc}}}\left({\cal U},{\cal W}\right)
\end{multline*}
By the induction hypothesis, the first sum is equal to
\[\sum_{\pi:\left({\cal U},\pi\right)\in{\cal PS}^{\prime}\left(p,q\right)}\sum_{\substack{\left({\cal W},\sigma\right)\succeq\left({\cal U},\pi\right)\\{\cal W}\in\left[{\cal U},{\cal V}\right)\\\textrm{\({\cal W}\) one bridge}}}\mu_{{\cal PS}^{\prime}}\left(\left({\cal U},\pi\right),\left({\cal W},\sigma\right)\right)\]
while the second sum is
\[\sum_{\pi:\left({\cal U},\pi\right)\in{\cal PS}^{\prime}\left(p,q\right)}\sum_{\substack{\sigma\in\left[\pi,{\cal V}\right)\\\textrm{\(\sigma\) more than one bridge}}}\mu\left(\pi,\sigma\right)\]
(since any \({\cal W}\) is greater than at most one preimage in \(\Pi^{-1}\left(\pi\right)\), where its contribution is given by Case~\ref{case: 1-n}, and if it is not greater than any preimage its contribution vanishes by Case~\ref{case: 1-n 0}; and no such \(\sigma\) is greater than \(\left({\cal U},\pi_{0}\right)\), so the outer summand \(\pi=\pi_{0}\) may be ignored).  The terms in the two sums may be partitioned by the element \(\left({\cal U},\pi\right)\in{\cal PS}^{\prime}\left(p,q\right)\) by which they are indexed in the outer sum.  For each such \(\left({\cal U},\pi\right)\), we find in the two sums all terms \(\left({\cal W},\sigma\right)\in\left[\left({\cal U},\pi\right),\left({\cal V},\rho_{0}\right)\right]\) except those where \({\cal W}={\cal V}\), i.e.\ those appearing in (\ref{equation: 1-1}).  Since
\[\sum_{\left({\cal W},\sigma\right)\in\left[\left({\cal U},\pi\right),\left({\cal V},\rho_{0}\right)\right]}\mu_{{\cal PS}^{\prime}}\left(\left({\cal U},\pi\right),\left({\cal W},\sigma\right)\right)=0\]
(since \({\cal U}\neq{\cal V}\) and hence \(\left({\cal U},\pi\right)\neq\left({\cal V},\rho_{0}\right)\)) the expression (\ref{equation: 1-1}) follows.

The poset \(\left[{\cal U},{\cal V}\right]\) is the product of the posets on the blocks of \({\cal V}\), and the contribution of every block except the bridge is given by (\ref{equation: disc noncrossing}), so we may restrict our attention to the case \({\cal V}=1\).

The \(\pi\) with \(\left({\cal U},\pi\right)\in{\cal PS}^{\prime}\left(p,q\right)\) for a fixed \({\cal U}\) are the \(\pi\in\Pi^{-1}\left({\cal U}\right)\) and \(\pi=\pi_{0}\).

In the former case, the \(\pi\in\Pi^{-1}\left({\cal U}\right)\) may be constructed by choosing an element from each end of the bridge as \(a_{1}\) and \(b_{1}\) in the notation of Lemma~\ref{lemma: technical}, part~\ref{item: bridge}.  Then the unique bridge of \(\mathrm{Kr}^{-1}\left(\pi\right)\) contains \(a_{1}\) and \(b_{1}\), and by Lemma~\ref{lemma: technical} part~\ref{item: outside faces}, it consists of the two orbits of \(\mathrm{Kr}^{-1}\left(\pi_{0}\right)\) containing these elements, while all other orbits of \(\mathrm{Kr}^{-1}\left(\pi\right)\) are equal to the orbits of \(\mathrm{Kr}^{-1}\left(\pi_{0}\right)\).  Thus choosing a \(\pi\in\Pi^{-1}\left({\cal U}\right)\) is equivalent to choosing a block of \(\Pi\left(\mathrm{Kr}^{-1}\left(\pi_{0}\right)\right)\) from each cycle of \(\tau\), each sharing at least one element with \(U_{0}\).  Since \(\mathrm{Kr}\left(\pi_{0}\right)=\tau^{-1}\mathrm{Kr}^{-1}\left(\pi_{0}\right)\tau\), this is finally equivalent to choosing a block of \(\Pi\left(\mathrm{Kr}\left(\pi_{0}\right)\right)\) from each cycle of \(\tau\), each containing at least one element of \(\tau^{-1}\left(U_{0}\right)\), whence the first summation in (\ref{equation: P 1-1}).  (We finally note that when \({\cal V}\neq 1\), any block of \(\mathrm{Kr}_{\rho_{0}}\left(\pi_{0}\right)\) intersecting \(\rho_{0}^{-1}\left(U_{0}\right)\) will be contained in the an orbit of \(\rho_{0}\) containing elements of \(U_{0}\), and hence in the bridge of \({\cal V}\), so the condition that the \(U_{1},U_{2}\) in the summation be in the bridge of \({\cal V}\) is redundant.)

For a fixed \(\pi\), the \(\left(1,\rho\right)\succeq\pi\) with \(\Pi\left(\rho\right)=1\) are those where \(\mathrm{Kr}^{-1}\left(\rho\right)=\left(a,b\right)\) where \(a\) and \(b\) are contained in the unique bridge of \(\mathrm{Kr}^{-1}\left(\pi\right)\) (see Case~\ref{case: 0-1}).  We let the bridge of \(\mathrm{Kr}^{-1}\left(\pi\right)\) be \(\bar{U}_{0}=\left(c_{1},\ldots,c_{r},d_{1},\ldots,d_{s}\right)\) where \(c_{1},\ldots,c_{r}\in\left[p\right]\) and \(d_{1},\ldots,d_{s}\in\left[p+1,p+q\right]\) (so \(a=c_{i}\) and \(b=d_{j}\) for some \(i\in\left[r\right]\) and \(j\in\left[s\right]\)).  We compute that the cycles of \(\mathrm{Kr}_{\mathrm{Kr}^{-1}\left(\pi\right)}\left(\left(a,b\right)\right)\) contained in this bridge are
\[\left(c_{1},\ldots,c_{i-1},d_{j},\ldots,d_{s}\right)\left(c_{i},\ldots,c_{r},d_{1},\ldots,d_{j-1}\right)\]
and the remaining cycles of \(\mathrm{Kr}_{\mathrm{Kr}^{-1}\left(\pi\right)}\left(\left(a,b\right)\right)\) are those of \(\mathrm{Kr}^{-1}\left(\pi\right)\).  Thus the contribution of the \(\mu_{{\cal PS}^{\prime}}\left(\pi,\mathrm{Kr}\left(\left(a,b\right)\right)\right)\) for all possible \(a\) and \(b\) is
\[\left(-1\right)^{r+s}\sum_{i=1}^{r}\sum_{j=1}^{s}C_{s+i-j-1}C_{r-i+j-1}\prod_{U\in\Pi\left(\mathrm{Kr}\left(\pi\right)\right)\setminus\left\{\bar{U}_{0}\right\}}\left(-1\right)^{\left|U\right|-1}C_{\left|U\right|-1}\textrm{.}\]
The value of the double summation is given by (\ref{equation: partition face}).

For each fixed \(\pi\), we also have \(\left(1,\tau\right)\succeq\pi\), whose contribution is given by (\ref{equation: PS easy}).  These two contributions give the first summation on the right-hand side in (\ref{equation: P 1-1}).

In the latter case where \(\left({\cal U},\pi\right)=\left({\cal U},\pi_{0}\right)\), the only \(\left(1,\rho\right)\in{\cal PS}^{\prime}\left(p,q\right)\) with \(\left(1,\rho\right)\succeq\left({\cal U},\pi_{0}\right)\) is \(\left(1,\tau\right)\).  The contribution is given by (\ref{equation: PS easy}), whence the final product on the right-hand side of (\ref{equation: P 1-1}).
\label{case: 1-1}
\end{case}
\begin{case}\({\cal U}\) has more than one bridge and \({\cal V}\) has exactly one bridge.

We first show by induction on \(\#\left({\cal U}\right)-\#\left({\cal V}\right)\) that
\begin{equation}
\mu_{{\cal P}_{\mathrm{nc}}}\left({\cal U},{\cal V}\right)=\sum_{\substack{\rho:\left({\cal V},\rho\right)\in{\cal PS}^{\prime}\left(p,q\right)\\\Pi^{-1}\left({\cal U}\right)\preceq\left({\cal V},\rho\right)}}\mu_{{\cal PS}^{\prime}}\left({\cal U},{\cal V}\right)\textrm{.}
\label{equation: n-1}
\end{equation}

In the base case, \({\cal U}\) has two bridges, and \({\cal V}\) is constructed from \({\cal U}\) by replacing them with their union.  In this case, \({\cal V}\) covers \({\cal U}\), so \(\mu_{{\cal P}_{\mathrm{nc}}}\left({\cal U},{\cal V}\right)=-1\).  There are exactly two ways of constructing a \(\rho\in\Pi^{-1}\left({\cal V}\right)\) with \(\Pi^{-1}\left({\cal U}\right)\preceq\rho\) (i.e.\ so that \(\rho\) does not satisfy annular-nonstandard condition~\ref{item: ans2} and the bridges of \(\Pi^{-1}\left({\cal U}\right)\) are noncrossing on the resulting cycle), which are the only elements in the interval \(\left({\cal U},{\cal V}\right)\subseteq{\cal PS}^{\prime}\left(p,q\right)\).  Thus \(\mu_{{\cal PS}^{\prime}}\left({\cal U},{\cal V}\right)=-1\), consistent with (\ref{equation: n-1}).

For the induction step,
\begin{multline*}
\mu_{{\cal P}_{\mathrm{nc}}}\left({\cal U},{\cal V}\right)=-\sum_{{\cal W}\in\left[{\cal U},{\cal V}\right)}\mu_{{\cal P}_{\mathrm{nc}}}\left({\cal U},{\cal W}\right)\\=-\sum_{\substack{{\cal W}\in\left[{\cal U},{\cal V}\right)\\\textrm{\({\cal W}\) one bridge}}}\mu_{{\cal P}_{\mathrm{nc}}}\left({\cal U},{\cal W}\right)-\sum_{\substack{{\cal W}\in\left[{\cal U},{\cal V}\right)\\\textrm{\({\cal W}\) more than one bridge}}}\mu_{{\cal P}_{\mathrm{nc}}}\left({\cal U},{\cal W}\right)\\=-\sum_{\substack{\left({\cal W},\sigma\right)\succeq{\cal U}\\{\cal W}\in\left[{\cal U},{\cal V}\right)\\\textrm{\({\cal W}\) one bridge}}}\mu_{{\cal PS}^{\prime}}\left({\cal U},\left({\cal W},\sigma\right)\right)-\sum_{\substack{\left({\cal W},\sigma\right)\succeq{\cal U}\\{\cal W}\in\left[{\cal U},{\cal V}\right)\\\textrm{\({\cal W}\) more than one bridge}}}\mu_{{\cal PS}^{\prime}}\left({\cal U},{\cal W}\right)
\end{multline*}
(the second sum on the right-hand side since \(\mu_{{\cal P}_{\mathrm{nc}}}\left({\cal U},{\cal W}\right)\) vanishes when \(\Pi^{-1}\left({\cal U}\right)\npreceq{\cal W}\)).  The terms in the two sums correspond to the \(\left({\cal W},\sigma\right)\in\left[{\cal U},{\cal V}\right]\) (the \(\sigma\) are noncrossing on \(\Pi^{-1}\left({\cal V}\wedge\tau\right)\) by Lemma~\ref{lemma: technical}, part~\ref{item: piecewise}) except those with \({\cal W}={\cal V}\), i.e.\ those in (\ref{equation: n-1}).

Since \(\mu_{{\cal P}_{\mathrm{nc}}}\left({\cal U},{\cal V}\right)\) is the product over the blocks of \({\cal V}\) of the M\"{o}bius function of the restriction to that block, we may restrict our attention to the case where \({\cal V}=1\).  As in Case~\ref{case: 0-1}, the \(\rho\in\Pi^{-1}\left(1\right)\) are those with \(\mathrm{Kr}^{-1}\left(\rho\right)=\left(a,b\right)\).  The \(\rho\succeq{\cal U}\) are those where \(\mathrm{Kr}^{-1}\left(\rho\right)\preceq\mathrm{Kr}^{-1}\left({\cal U}\right)\), i.e.\ where \(a\) and \(b\) are contained in the same bridge of \(\mathrm{Kr}^{-1}\left({\cal U}\right)\).

As in Case~\ref{case: 1-1}, we let the bridge of \(\mathrm{Kr}^{-1}\left({\cal U}\right)\) containing \(a\) and \(b\) be \(U_{0}=\left(c_{1},\ldots,c_{r},d_{1},\ldots,d_{s}\right)\) where \(c_{1},\ldots,c_{r}\in\left[p\right]\) and \(d_{1},\ldots,d_{s}\in\left[p+1,p+q\right]\) and \(a=c_{i}\) and \(b=d_{j}\) for some \(i,j\).  Again, the cycles of \(\mathrm{Kr}_{\mathrm{Kr}^{-1}\left({\cal U}\right)}\left(\left(a,b\right)\right)\) contained in this bridge are
\[\left(c_{1},\ldots,c_{i-1},d_{j},\ldots,d_{s}\right)\left(c_{i},\ldots,c_{r},d_{1},\ldots,d_{j-1}\right)\]
and the remaining cycles of \(\mathrm{Kr}_{\mathrm{Kr}^{-1}\left({\cal U}\right)}\left(\left(a,b\right)\right)\) are those of \(\mathrm{Kr}^{-1}\left({\cal U}\right)\).  The contribution of the \(\rho\) with \(a\) and \(b\) in this particular bridge is
\[\left(-1\right)^{r+s}\sum_{i=1}^{r}\sum_{j=1}^{s}C_{s+i-j-1}C_{r-i+j-1}\prod_{U\in\Pi\left(\mathrm{Kr}\left({\cal U}\right)\right)\setminus\left\{U_{0}\right\}}\left(-1\right)^{\left|U\right|-1}C_{\left|U\right|-1}\]
where the formula for the double sum is given by (\ref{equation: partition face}).  The sum of the contributions of the \(\mu_{{\cal PS}^{\prime}}\left({\cal U},\rho\right)\) is given by the first term on the right-hand side of (\ref{equation: P n-1}).  (For general \({\cal V}\) and \(\rho\in\Pi^{-1}\left({\cal V}\right)\), any bridge \(\mathrm{Kr}^{-1}_{\rho}\left({\cal U}\right)\) is in the bridge of \({\cal V}\), so the restriction that \(U_{0}\) must be in the bridge of \({\cal V}\) is redundant.)

The \(\left(1,\rho\right)\in{\cal PS}^{\prime}\left(p,q\right)\) include a single element with a nontrivial block, \(\left(1,\tau\right)\).  The contribution of \(\mu_{{\cal PS}^{\prime}}\left({\cal U},1\right)\) is the second term on the right-hand side of (\ref{equation: P n-1}).
\label{case: n-1}
\end{case}
\end{proof}

We give here values of (\ref{equation: two bridges}) and (\ref{equation: partition face}) for small values of \(p\) and \(q\):

\begin{table}[h]
\centering
\begin{tabular}{r|ccccc}
&2&3&4&5&6\\
\hline
2&1&-4&14&-48&165\\
3&-4&18&-68&246&-880\\
4&14&-68&271&-1020&3762\\
5&-48&246&-1020&3958&-14956\\
6&165&-880&3762&-14956&57638\\
\end{tabular}
\caption{Small values of the expression (\ref{equation: two bridges}), representing the sum of possible contributions of two bridges with a total of \(r\) elements on one end and \(s\) elements on the other.}
\end{table}

\begin{table}[h]
\centering
\begin{tabular}{r|ccccc}
&1&2&3&4&5\\
\hline
1&1&-2&5&-14&42\\
2&-2&6&-18&56&-180\\
3&5&-18&60&-200&675\\
4&-14&56&-200&700&-2450\\
5&42&-180&675&-2450&8820\\
\end{tabular}
\caption{Small values of the expression (\ref{equation: partition face}), representing the contribution to \(\mu_{{\cal P}_{\mathrm{nc}}}\) of certain bridges with \(r\) elements on one end and \(s\) elements on the other.}
\end{table}

\section{Acknowledgments}

The author would like to thank Feodor Kogan for several useful discussions leading to this paper.

\bibliography{paper}
\bibliographystyle{plain}

\end{document}